\documentclass[english]{amsart}

\usepackage{adjustbox}
\usepackage[latin9]{inputenc}
\usepackage{hyperref}
\usepackage{lscape}
\usepackage{slashbox}
\usepackage{mathrsfs}  
\usepackage{enumerate}  
\makeatletter
 \theoremstyle{plain}
\newtheorem{thm}{Theorem}[section]
\theoremstyle{plain}
  \newtheorem{prop}[thm]{Proposition}
\theoremstyle{plain}
 \newtheorem{lemma}[thm]{Lemma}
\theoremstyle{plain}

\theoremstyle{plain}
\newtheorem{cor}[thm]{Corollary}
\theoremstyle{definition}
  \newtheorem{defn}[thm]{Definition}
\theoremstyle{definition}
  
 \theoremstyle{definition}

\theoremstyle{remark}
\newtheorem{rmk}[thm]{Remark}
\numberwithin{equation}{section}
 
\usepackage{amsmath}
\usepackage{amssymb}
\usepackage{amscd}
\usepackage{amsthm}
\usepackage{array,multirow}
\usepackage[arrow,matrix,tips,all,cmtip,poly]{xy}

\usepackage{stmaryrd}
\usepackage{url}
\usepackage{color}
\usepackage{caption}
\usepackage{float}

\newcommand{\Z}{\mathbb{Z}}

\newcommand{\Q}{\mathbb{Q}}

\newcommand{\C}{\mathbb{C}}

\newcommand{\F}{\mathbb{F}}

\newcommand{\N}{\mathbb{N}}

\newcommand{\HT}{\mathrm{HT}}

\newcommand{\A}{\mathbb{A}}

\newcommand{\fE}{\mathfrak{e}}
\newcommand{\fF}{\mathfrak{f}}

\newcommand{\fM}{\mathfrak{M}}

\newcommand{\fm}{\mathfrak{m}}

\newcommand{\bF}{\mathbb{F}}

\newcommand{\bk}{\mathbf{k}}

\newcommand{\bZ}{\mathbb{Z}}

\newcommand{\nr}{\mathrm{nr}}

\newcommand{\cE}{\mathcal{E}}

\newcommand{\cM}{\mathcal{M}}

\newcommand{\cO}{\mathcal{O}}

\newcommand{\Id}{\mathrm{id}}
\newcommand{\Gal}{\mathrm{Gal}}
\newcommand{\Hom}{\mathrm{Hom}}

\newcommand{\Res}{\mathrm{Res}}
\newcommand{\Ind}{\mathrm{Ind}}

\newcommand{\SL}{\mathrm{SL}}
\newcommand{\GL}{\mathrm{GL}}

\newcommand{\PGL}{\mathrm{PGL}}

\newcommand{\Spec}{\mathrm{Spec}\ }

\newcommand{\Frob}{\mathrm{Frob}}

\newcommand{\id}{\mathrm{id}}
\newcommand{\isom}{\cong}

\newcommand{\tld}[1]{\widetilde{#1}}

\newcommand{\JH}{\mathrm{JH}}

\newcommand{\rbar}{\overline{r}}
\newcommand{\rhobar}{{\overline{\rho}}}

\newcommand{\sigmabar}{\overline{\sigma}}

\newcommand{\Spf}{\mathrm{Spf}}

\newcommand{\rad}{\mathrm{rad}}
\newcommand{\cosoc}{\mathrm{cosoc}}

\newcommand{\defeq}{\stackrel{\textrm{\tiny{def}}}{=}}

\newcommand{\ovl}[1]{\overline{#1}}

\newif\iffinalrun
\iffinalrun
  \newcommand{\mar}[1]{}
\else
  \newcommand{\mar}[1]{\marginpar{\raggedright\tiny #1}}

\DeclareMathOperator{\et}{et}

\DeclareMathOperator{\Fil}{Fil}
\DeclareMathOperator{\Mat}{Mat}

\DeclareMathOperator{\gr}{gr}

\newcommand{\ra}{\rightarrow}

\newcommand{\into}{\hookrightarrow}
\newcommand{\surj}{\twoheadrightarrow}

\title{Multiplicity one for wildly ramified representations}
\author{Daniel Le}
\begin{document}

\maketitle

\begin{abstract}
Let $F$ be a totally real field in which $p$ is unramified.
Let $\rbar: G_F \ra \GL_2(\overline{\F}_p)$ be a modular Galois representation which satisfies the Taylor--Wiles hypotheses and is generic at a place $v$ above $p$.
Let $\fm$ be the corresponding Hecke eigensystem.
We show that the $\fm$-torsion in the mod $p$ cohomology of Shimura curves with full congruence level at $v$ coincides with the $\GL_2(k_v)$-representation $D_0(\rbar|_{G_{F_v}})$ constructed by Breuil and Pa\v{s}k\={u}nas.
In particular, it depends only on the local representation $\rbar|_{G_{F_v}}$, and its Jordan--H\"older factors appear with multiplicity one.
This builds on and extends work of the author with Morra and Schraen and, independently, Hu--Wang, which proved these results when $\rbar|_{G_{F_v}}$ was additionally assumed to be tamely ramified.
The main new tool is a method for computing Taylor--Wiles patched modules of integral projective envelopes using multitype tamely potentially Barsotti--Tate deformation rings and their intersection theory.
\end{abstract}

\section{Introduction}

Let $F/\Q$ be a totally real field which is unramified at a rational prime $p$.
Let $\F$ be a finite extension of $\F_p$.
Suppose that $\rbar: G_F \ra \GL_2(\F)$ is a Galois representation occuring in the $\F$-cohomology of a Shimura curve $X_{/F}$ with corresponding Hecke eigensystem $\fm$ (see \S \ref{sec:main}).
Suppose that the corresponding quaternion algebra splits at $p$.
Let $v$ be a place of $F$ dividing $p$, let $K^v$ be a compact open subgroup of $(D\otimes_F \A_F^{\infty,v})^\times$ and $K_v(n)$ the $n$-th principal congruence subgroup at $v$.
One expects that the analogues of the mod $p$ local Langlands correspondence for $\GL_2(\Q_p)$ and mod $p$ local-global compatibility for $\GL_2(\Q)$ describe the $\GL_2(F_v)$-representation
\[\pi' = \Hom_{G_F}(\rbar,\varinjlim_n H^1(X(K^vK_v(n)),\F)[\fm_{\rbar}])\]
in the completed cohomology of $X$, at least up to multiplicities, in terms of $\rhobar \defeq \rbar|_{G_{F_v}}$.
In fact, we study a related representation $\pi = (M^{\mathrm{min}})^*$ (see \S \ref{sec:main}), which is minimal with respect to multiplicities.
Such analogues are unknown at present, although \cite{breuil,EGS} show that if $\rbar$ satisfies the usual Taylor--Wiles hypotheses and $\rhobar$ is generic, then $\pi$ contains one of infinitely many $\GL_2(F_v)$-representations constructed by \cite{BP}.
The idea, as explained in \cite{breuil}, behind the constructions of \cite{BP} is that if one can show that the restriction of $\pi$ to the maximal compact subgroup $\GL_2(\cO_{F_v})$ satisfies certain multiplicity one properties, then $\pi$ must contain a Diamond diagram of the form $D(\rhobar,\iota)$.
These multiplicity one properties, which one might view as minimalist conjectures for multiplicities, were established in \cite{EGS}.

That the family of representations containing a diagram $D(\rhobar,\iota)$ is infinite is unfortunate and warrants further investigation of $\pi$.
One part of a Diamond diagram $D(\rhobar,\iota)$ is a $\GL_2(k_v)$-representation denoted $D_0(\rhobar)$, which is a subrepresentation of  $\pi|_{\GL_2(\cO_{F_v})}$ (see \cite[Proposition 9.3]{breuil}), and thus a subrepresentation of the invariants of $\pi$ under the first principal congruence subgroup $K_v(1)$ of $\GL_2(\cO_{F_v})$.
Our main result is the following.

\begin{thm}[Corollary \ref{cor:main}]\label{intro:mainthm}
If $\rbar$ satisfies the Taylor--Wiles hypotheses and $\rhobar$ is generic $($see Definition \ref{def:gen}$)$, then the $\GL_2(k_v)$-representation $\pi^{K_v(1)}$ is isomorphic to $D_0(\rhobar)$.
In particular, it only depends on $\rhobar$ and is multiplicity free.
\end{thm}

\noindent One can view this result as showing that $\pi$ satisfies a minimality property: $\pi^{K_v(1)}$ is as small as possible.
A similar result has been announced by Hu--Wang.

The main tool in the proof of Theorem \ref{intro:mainthm} is the Taylor--Wiles patching method.
Diamond and Fujiwara \cite{D,F} discovered that the Cohen--Macaulay property of patched modules could be combined with local algebra results of Auslander, Buchsbaum, and Serre to rederive and generalize mod $p$ multiplicity one results of Mazur for modular forms with level away from $p$.
\cite{EGS} proved similar results for modular forms with level at $p$ by introducing two gluing methods to calculate patched modules from smaller ones to which the Diamond--Fujiwara trick applied.
The first method is a version of Nakayama's lemma and uses the submodule structure of mod $p$ reductions of Deligne--Lusztig representations.
The second method combines the submodule structure above with the intersection theory of special fibers of tamely potentially Barsotti--Tate deformation rings.

When $\rhobar$ is tamely ramified, \cite{HW,LMS} show that the patched modules of projective envelopes of irreducible $\F[\GL_2(k_v)]$-modules are cyclic modules by describing the submodule structure of these projective envelopes and using the Nakayama method of \cite{EGS} (cf.~ Proposition \ref{prop:oldcyc}).
However, the gluing methods of \cite{EGS} are insufficient when $\rhobar$ is wildly ramified.
Indeed, these methods only glue together characteristic $p$ patched modules, but when $\rhobar$ is wildly ramified there is more than one isomorphism class of $\F[\GL_2(k_v)]$-modules satisfying the multiplicity one properties for $\pi^{K_v(1)}$ established in \cite{EGS}.

We introduce a variant of the intersection theory method of \cite{EGS}, which uses the intersection theory of integral tamely potentially Barsotti--Tate deformation rings.
Let $W(\F)$ denote the Witt vectors of $\F$.
The first step (Proposition \ref{prop:oldcyc}) is to show that the methods of \cite{EGS} still apply to certain quotients of generic $W(\F)[\GL_2(k_v)]$-projective envelopes (which are projective envelopes in the abelian category of $W(\F)[\GL_2(k_v)]$-modules generated by lattices in some fixed set of Deligne--Lusztig representations).
If such a quotient is reducible rationally, then it can be written as a submodule of the direct sum of two smaller quotients with $p$-torsion cokernel (see Proposition \ref{prop:exseq}).
This reflects a kind of transversality: while these subcategories do not give a direct product decomposition of the category of $W(\F)[\GL_2(k_v)]$-modules, if two subquotients of lattices in two distinct Deligne--Lusztig representations are isomorphic, they must be $p$-torsion.
By exactness of patching and this exact sequence, it turns out that the patched modules of $W(\F)[\GL_2(k_v)]$-projective envelopes are then determined by the patched modules of these quotients (this depends crucially on the fact that all such patched modules turn out to be cyclic).

It remains to actually compute these patched modules using intersection theory in a multitype Barsotti--Tate framed deformation space, which we define to be the Zariski closure in the unrestricted framed deformation space of $\rhobar$ of potentially Barsotti--Tate Galois representations with tame inertial type in some fixed set.
That the resulting patched module is cyclic comes from the fact that the multitype Barsotti--Tate deformation rings exhibit a similar kind of transversality: two lattices in potentially Barsotti--Tate Galois representations of two distinct generic tame inertial types can be congruent modulo $p$, but never modulo $p^2$.

We now give a brief overview of the following sections.
In \S 2, we generalize some of the results of \cite{LMS} and prove the key result (Proposition \ref{prop:exseq}) gluing integral projective envelopes from their quotients.
In \S 3, we define and calculate multitype Barsotti--Tate deformation rings---this is the other key technical input.
To compare Kisin modules for varying tame types, it is much more convenient to choose eigenbases for Kisin modules which are not always gauge bases in the sense of \cite[\S 7.3]{EGS}.
This requires generalizing \cite[Theorem 4.1]{LLLM}.
The main result, Theorem \ref{thm:defring}, of this section computes some multitype Barsotti-Tate framed deformation spaces.
In \S 4, we calculate the abstract patched modules of projective envelopes using the Nakayama method and our integral intersection theory method.
In \S 5, we apply the results of \S 4 to the cohomology of Shimura curves using the Taylor--Wiles method.

\subsection{Acknowledgments}

Lemma \ref{lemma:ca} originally appeared in \cite{LLM}, and we thank Bao Le Hung and Stefano Morra for allowing us to reproduce it here.
The idea to use multitype Barsotti--Tate deformation rings grew out of the joint work (\cite{LLM}).
We thank Bao Le Hung and Stefano Morra for their collaboration and other useful discussions on Kisin modules and \'etale $\varphi$-modules.
We thank the anonymous referee for many useful suggestions, comments, and corrections.
The debt owed to the work of Christophe Breuil, Matthew Emerton, Toby Gee, Vytautas Pa\v{s}k\={u}nas, and David Savitt will be clear to the reader.

The author was supported by the Simons Foundation under an AMS-Simons travel grant, by the National Science Foundation under the Mathematical Sciences Postdoctoral Research Fellowship No.~ 1703182, and by the Centre International de Rencontres Math\'ematiques under the Research in Pairs program No.~ 1877.
We thank CIRM for providing hospitality and excellent working conditions while part of this work was carried out.

\subsection{Notation}\label{subsec:not}

If $F$ is any field, we write $\overline{F}$ for a separable closure of $F$ and $G_F:= \mathrm{Gal}(\overline{F}/F)$ for the absolute Galois group of $F$.

Let $f\in \N$ and $q = p^f$.
Let $\cO_K$ be the Witt vectors $W(\F_q)$ of $\F_q$.
Let $K = \cO_K[p^{-1}]$ be the unramified extension of $\Q_p$ of degree $f$.
Let $E$ be an extension of $K$ with ring of integers $\cO$, uniformizer $\varpi$, and residue field $\F$.
This induces embeddings $\cO_K\into \cO$ and $\iota_0: \F_q \into \F$.
For $i\in \Z/f$, let $\iota_i = \iota_0\circ \varphi^i$ be the $i$-th Frobenius twist of $\iota_0$.
We fix an embedding $\F \into \overline{\F}_q$.
We will denote by $(\cdot)^*$ the $\F$-linear dual, and by $(\cdot)^\vee$ the contragredient of a representation.

Let $G$ (resp.~$G^{\mathrm{der}}$) be the algebraic group $\Res_{\F_q/\F_p} \GL_2$ (resp.~$\Res_{\F_q/\F_p} \SL_2$), and let $T\subset G$ (resp.~$T^{\mathrm{der}} \subset G^{\mathrm{der}}$) be the diagonal torus.
Let $X^*(T)$ (resp.~$X^*(T^{\mathrm{der}})$) denote the group of characters of $T$ (resp.~ $T^{\mathrm{der}}$).
Let $X_*(T)$ and $X_*(T^{\mathrm{der}})$ similarly denote groups of cocharacters.
By the embeddings $\iota_i$, $X^*(T)$ is identified with $X^*(T \times_{\F_p} \F) \cong X^*(\prod_{i\in \Z/f} \mathbb{G}_m^2)$, which is identified with $(\Z^2)^{\Z/f}$ in the usual way.
A similar identification for $X_*(T)$ is made.
For a character $\mu\in X^*(T)$, we write $\mu_i$ as the $i$-th factor of $\mu$ so that $\mu = \sum_{i\in \Z/f} \mu_i$.

Let $\eta^{(i)} \in X^*(T)$ (resp.~ $\alpha^{(i)} \in X_*(T)$) be the dominant fundamental character (resp.~ the positive coroot) represented by $(1,0)$ (resp.~ $(1,-1)$) in the $i$-th factor and $0$ elsewhere.
Let $\eta = \sum_{i\in \Z/f} \eta^{(i)}$.
Let $\omega^{(i)}$ be the restriction of $\eta^{(i)}$ to $T^{\mathrm{der}}$.

Let $W$ be the Weyl group of $G$ and $G^{\mathrm{der}}$, which is similarly identified with $S_2^{\Z/f}$.
Here, $S_2$ denotes the permutation group on two elements.
We denote the trivial element of $S_2$ by $\Id$.
Then $W$ acts naturally on $X^*(T)$ and $X^*(T^{\mathrm{der}})$.
Let $\pi$ be the automorphism of $X^*(T)$ and $W$ which acts by a shift so that $\pi(x)_i = x_{i-1}$.
Then the action on $X^*(T)$ induced by the relative Frobenius morphism on $T$ is given by $p\pi^{-1}$, while the action of the relative Frobenius on $W$ is given by $\pi$.

For a dominant character $\mu\in X^*(T)$ we write $V(\mu)$ for the Weyl module for $G$ defined in \cite[II.2.13(1)]{JantzenBook}. It has a unique simple $G$-quotient $L(\mu)$. 
If $\mu = \sum_i \mu_i$ is $p$-restricted (i.e. $0\leq \langle \mu,\alpha^{(i)}\rangle \leq p$ for all $i$), then $L(\mu) = \otimes_i L(\mu_i)$ by the Steinberg tensor product theorem as in \cite[Theorem 3.9]{Herzig}.
Let $F(\mu)$ be the restriction of $L(\mu)$ to $\GL_2(\F_q)$, which remains irreducible by \cite[A.1.3]{Herzig}.
Every irreducible $\GL_2(\F_q)$-representation is of this form, and we call such a representation a {\it Serre weight}.
Note that $F(\mu) \cong F(\lambda)$ if and only if $\mu \cong \lambda \mod{(p-\pi)X^0(T)}$, where $X^0(T)$ is the kernel of the restriction map $X^*(T)\ra X^*(T^{\mathrm{der}})$.

Recall that to a pair $(s,\lambda)\in W\times X^*(T)$, \cite[Lemma 4.2]{Herzig} attaches a (virtual) representation of $\GL_2(\F_q)$, which we denote $R_s(\lambda)$.
In each use below, $R_s(\lambda)$ will in fact denote a true representation.

An \emph{inertial type} for a local field $L$ is a continuous $E$-representation $\tau$ of the inertial subgroup $I_L$, whose action factors through a finite quotient and can be extended to $G_L$.
For our purposes, all inertial types will be two-dimensional.
In this case, Henniart's \cite[Annexe A]{BM} attaches to $\tau$ a smooth irreducible finite-dimensional $\GL_2(\cO_L)$-representation $\sigma(\tau)$ over $E$ (see also \cite[\S 1.9]{EGS}).
%Let $\sigma(\tau)$ be the contragredient of this representation (this is dual to the notation in \emph{loc.~cit.}).
We call the association of $\tau$ and $\sigma(\tau)$ the inertial local Langlands correspondence.
An inertial type $\tau$ is called \emph{tame} if $\tau$ factors through the tame quotient of $I_L$.
The tame inertial types are exactly those $\tau$ such that $\sigma(\tau)$ factors through $\GL_2(k_L)$ where $k_L$ is the residue field of $L$.

For any characteristic $0$ field $F$, let $\varepsilon: G_F \ra \Z_p^\times \subset \cO^\times$ denote the $p$-adic cyclotomic character and $\ovl{\varepsilon}$ denote its reduction modulo $\varpi$.
We now let $F$ be $K$.
Let $\C_p(i)$ denote $\varepsilon^i \otimes_E \C_p$, where the tensor product is over any embedding $E \into \C_p$.
Let $\rho: G_K \ra \GL(V)$ be a continuous representation over $E$.
For each embedding $\kappa: E \into \C_p$, let $\HT_\kappa(V)$ be the multiset of integers such that $-i$ appears with multiplicity $\dim_{\C_p} (V\otimes_\kappa \C_p(i))^{G_K}$.
Then in particular $\HT_\kappa(\varepsilon) = \{ 1 \}$ for all embeddings $\kappa$.
We say that a two-dimensional representation $V$ is (\emph{potentially}) \emph{Barsotti--Tate} if $V$ is (potentially) crystalline with $\HT_\kappa(V) = \{0,1\}$ for all embeddings $\kappa$.
If $\tau$ is an inertial type, we say that $V$ is potentially Barsotti--Tate of type $\tau$ if the action of $I_K$ on the potentially crystalline Dieudonn\'e module of $V$ is isomorphic to $\tau$.

\section{Quotients of generic $\GL_2(\F_q)$-projective envelopes} \label{sec:proj}

Suppose that $\mu \in X^*(T)$ and that $1\leq \langle \mu-\eta,\alpha^{(i)}\rangle < p-2$ for all $i\in \Z/f$.
Let $\sigma$ be $F(\mu-\eta)$.
Let $\tld{R}_\mu$ (resp.~ $R_\mu$) be the projective $\cO_K[\GL_2(\F_q)]$-envelope (resp.~ the projective $\F_q[\GL_2(\F_q)]$-envelope) of $\sigma$.
Let $S$ be the set $\{\pm\omega^{(i)}\}_i$ and let $I$ be a subset of $S$.
Recall from \cite[Definition 3.5]{LMS} that (with respect to $\mu$) we attach to a subset $J \subset S$ a Serre weight $\sigma_J$.
Let $R_{\mu,I}$ be the universal object among quotients of $R_\mu$ that do not contain $\sigma_{\{\omega\} }$ as a Jordan--H\"older factor for all $\omega$ in $I$.
Recall from \cite[\S 3]{LMS} that there is a filtration $\Fil^\bk$ on $R_\mu$ which induces a filtration $\Fil^\bk$ on $R_{\mu,I}$.
Similarly, we can construct a filtration $\Fil^k_\otimes = \sum_{|\bk|=k} \Fil^{\bk}$ on $R_\mu$ and $R_{\mu,I}$.
Let $W_{\bk,I}$ be $\gr^\bk R_{\mu,I}$.

\begin{prop}\label{prop:Wk}
We have an isomorphism $W_{\bk,I} \cong \oplus_{J \subset S, \bk(J) = \bk, J \cap I = \emptyset} \sigma_J$.
\end{prop}
\begin{proof}
This follows from \cite[Proposition 3.6 and Theorem 3.14]{LMS}.
\end{proof}

If $I$ is a subset of $S$ such that $I \cap \{\pm \omega^{(i)}\}$ has size at most one for all $i$, let $T_{\sigma,I}$ be the set of Deligne--Lusztig representations over $K$ of the form $R_w(\mu-w\eta)$ where $w_i = \Id$ (resp.~ $w_i \neq \Id$) if $\omega^{(i)} \in I$ (resp.~ $-\omega^{(i)} \in I$).
Fix an embedding $\tld{R}_\mu \into \oplus_{\sigma(\tau) \in T_{\sigma,\emptyset}} \sigma(\tau)$.
Let $\tld{R}_{\mu,I}$ be the quotient of $\tld{R}_\mu$ isotypic for the set $T_{\sigma,I}$ (which does not depend on the above embedding).
Note that $\tld{R}_{\mu,\emptyset}$ is equal to $\tld{R}_\mu$.

\begin{prop}\label{prop:projred}
The reduction of $\tld{R}_{\mu,I}$ modulo $p$ is $R_{\mu,I}$.
\end{prop}
\begin{proof}
For each $\omega\in I$, $\sigma_{\{\omega\} } \notin \JH(\sigmabar(\tau))$ for all $\sigma(\tau) \in T_{\sigma,I}$.
Thus, there is a canonical quotient map $R_{\mu,I} \ra \overline{R}_{\mu,I}$, where $\overline{R}_{\mu,I}$ is the reduction of $\tld{R}_{\mu,I}$.
By Proposition \ref{prop:Wk}, $R_{\mu,I}$ has length $2^{2f-\#I}$.
Since $\overline{R}_{\mu,I}$ is the reduction of a lattice in the direct sum of $2^{f-\#I}$ types, each of whose reduction has length $2^f$ (see \cite{diamond}), it also has length $2^{2f-\#I}$.
Since both objects have the same length, this surjection must be an isomorphism.
\end{proof}

Again, let $I\subset S$.
Let $W_{\bk,\bk+1,I}$ be $\Fil^\bk R_{\mu,I}/(\Fil^{k+2}_\otimes R_{\mu,I} \cap \Fil^\bk R_{\mu,I})$.
Note that $W_{\bk,\bk+1,I}$ is multiplicity free since $W_{\bk,\bk+1,\emptyset}$ (which is $W_{\bk,\bk+1}$ in \cite[\S 3]{LMS}) is by \cite[Proposition 3.6 and Lemma 3.7]{LMS}.

\begin{prop}\label{prop:ext}
Suppose that $J \subset J'$, $\#J'\setminus J = 1$, and $J' \cap I = \emptyset$.
Let $\bk$ and $\bk'$ be $\bk(J)$ and $\bk(J')$, respectively.
Then there is a subquotient of $W_{\bk,\bk+1,I}$ which is the unique up to isomorphism nontrivial extension of $\sigma_J$ by $\sigma_{J'}$.
\end{prop}
\begin{proof}
This follows immediately from Proposition \ref{prop:Wk} and \cite[Proposition 3.8]{LMS}.
\end{proof}

\begin{prop}\label{prop:exseq}
Suppose that the size of $I \cap \{\pm\omega^{(i)}\}$ is at most one for all $i$ and that $I \cap \{\pm\omega^{(j)}\} = \emptyset$ for some $j$.
Then there is an exact sequence
\begin{equation}\label{eqn:exact}
0 \ra \tld{R}_{\mu,I} \ra \tld{R}_{\mu,I\cup \{\omega^{(j)}\} } \oplus \tld{R}_{\mu,I\cup \{-\omega^{(j)}\} } \ra R_{\mu,I\cup \{\pm\omega^{(j)}\} } \ra 0,
\end{equation}
where the second (resp.~ third) map is the sum (resp.~ difference) of the natural projections.
\end{prop}
\begin{proof}
The second map of (\ref{eqn:exact}) is clearly injective since it is after inverting $p$ and $\tld{R}_{\mu,I}$ is $\cO_K$-flat.
We claim that the cokernel of this map is $p$-torsion.
Let $\sigma_{\{\omega^{(j)}\} } = F(\mu'-\eta)$ and consider a map $\tld{R}_{\mu'} \ra \tld{R}_{\mu,I}$ such that the composition with the projection 
\[
\tld{R}_{\mu,I} \surj R_{\mu,I} \surj R_{\mu,I}/\Fil^2_\otimes R_{\mu,I}
\]
is nonzero.
The composition of $\tld{R}_{\mu'} \ra \tld{R}_{\mu,I}$ with the natural surjection $\tld{R}_{\mu,I} \surj \tld{R}_{\mu,I\cup \{\omega^{(j)}\} }$ is zero since $\sigma_{\{\omega^{(j)}\} }\notin \JH(R_{\mu,I\cup \{\omega^{(j)}\} })$.
\begin{lemma}\label{lemma:containsp}
The image of the composition $\tld{R}_{\mu'} \ra \tld{R}_{\mu,I}$ with the natural surjection $\tld{R}_{\mu,I} \surj \tld{R}_{\mu,I\cup \{-\omega^{(j)}\} }$ contains $p\tld{R}_{\mu,I\cup \{-\omega^{(j)}\} }$.
\end{lemma}
\noindent With Lemma \ref{lemma:containsp} and its analogue for $\tld{R}_{\mu,I\cup \{\omega^{(j)}\} }$, we would see that the image of 
\[
\tld{R}_{\mu,I} \ra \tld{R}_{\mu,I\cup \{\omega^{(j)}\} } \oplus \tld{R}_{\mu,I\cup \{-\omega^{(j)}\} }
\]
contains $p\tld{R}_{\mu,I\cup \{\omega^{(j)}\} } \oplus p\tld{R}_{\mu,I\cup \{-\omega^{(j)}\} }$, establishing our claim.
\begin{proof}[Proof of Lemma \ref{lemma:containsp}]
Fix a map $\tld{R}_\mu\ra \tld{R}_{\mu'}$ such that the composition with the projection to $R_{\mu'}/\Fil^2_\otimes R_{\mu'}$ is nonzero.
It suffices to show that the image, denoted $Q$, of the composition of $\tld{R}_\mu\ra \tld{R}_{\mu'}$ with the above $\tld{R}_{\mu'} \ra \tld{R}_{\mu,I} \surj \tld{R}_{\mu,I\cup \{-\omega^{(j)}\} }$ is $p\tld{R}_{\mu,I\cup \{-\omega^{(j)}\} }$.
On the one hand, we see that $Q$ is in $p\tld{R}_{\mu,I\cup \{-\omega^{(j)}\} }$ by reducing modulo $p$ and using Propositions \ref{prop:projred} and \ref{prop:ext}.
Let $\sigma(\tau)$ be a Jordan--H\"older factor of $\tld{R}_{\mu,I}[p^{-1}]$ and let $\sigma^\circ(\tau) \subset \sigma(\tau)$ be the unique lattice up to homothety with cosocle isomorphic to $\sigma$ (see \cite[Lemma 4.1.1]{EGS}).
Fix a surjection from $\tld{R}_{\mu,I}$ to $\sigma^\circ(\tau)$.
By reducing mod $p$, we see that the image of the composition of $\tld{R}_{\mu'} \ra \tld{R}_{\mu,I}$ with this surjection is a saturated lattice $\sigma^{\circ\circ}(\tau)$ with cosocle $\sigma_{\{\omega^{(j)}\} }$.
Similarly, the image of $Q$ under this surjection is a saturated lattice in $\sigma^{\circ\circ}(\tau)$ with cosocle isomorphic to $\sigma$.
This lattice is $p \sigma^\circ(\tau)$ by \cite[Theorem 5.1.1]{EGS}.
Thus, the composition $Q \subset p\tld{R}_{\mu,I\cup \{-\omega^{(j)}\} } \surj p \sigma^\circ(\tau)$ is an isomorphism upon taking cosocles.
We see that $Q$ must be equal to $p\tld{R}_{\mu,I\cup \{-\omega^{(j)}\} }$.
\end{proof}

Let $R$ be the cokernel of the second map in (\ref{eqn:exact}), which is $p$-torsion by our first claim.
Then the exact sequence
\begin{equation}\label{eqn:char0}
0 \ra \tld{R}_{\mu,I} \ra \tld{R}_{\mu,I\cup \{\omega^{(j)}\} } \oplus \tld{R}_{\mu,I\cup \{-\omega^{(j)}\} } \ra R \ra 0
\end{equation}
induces an exact sequence
\begin{equation}\label{eqn:charp}
R_{\mu,I} \ra R_{\mu,I\cup \{\omega^{(j)}\} } \oplus R_{\mu,I\cup \{-\omega^{(j)}\} } \ra R \ra 0
\end{equation}
by Proposition \ref{prop:projred}.
By taking cosocles, (\ref{eqn:charp}) induces an exact sequence
\begin{equation}\label{eqn:cosoc}
\cosoc R_{\mu,I} \ra \cosoc R_{\mu,I\cup \{\omega^{(j)}\} } \oplus \cosoc R_{\mu,I\cup \{-\omega^{(j)}\} } \ra \cosoc R \ra 0
\end{equation}
Note that $\cosoc R_{\mu,I}$, $\cosoc R_{\mu,I\cup \{\omega^{(j)}\} }$, and $\cosoc R_{\mu,I\cup \{-\omega^{(j)}\} }$ are all isomorphic to $\sigma$ and that the composition of first map of (\ref{eqn:cosoc}) with either projection is nonzero.
Thus $\cosoc R$ is isomorphic to $\sigma$ and the restriction of the second map of (\ref{eqn:cosoc}) to either summand is nonzero.
We conclude that the restriction of the second map in (\ref{eqn:charp}) to either summand is surjective.
By definition, the maximal representation which is a quotient of both $R_{\mu,I\cup \{\omega^{(j)}\} }$ and $R_{\mu,I\cup \{-\omega^{(j)}\} }$ is $R_{\mu,I\cup \{\pm\omega^{(j)}\} }$.
Thus, there is a surjection $R_{\mu,I\cup \{\pm\omega^{(j)}\} }\surj R$.
On the other hand, it is easy to see that the composition $R_{\mu,I} \ra R_{\mu,I\cup \{\omega^{(j)}\} } \oplus R_{\mu,I\cup \{-\omega^{(j)}\} } \ra R_{\mu,I\cup \{\pm\omega^{(j)}\} }$ is zero, where the second map is the difference of the natural projections.
Thus, there is a surjection $R \surj R_{\mu,I\cup \{\pm\omega^{(j)}\} }$.
Since $R$ and $R_{\mu,I\cup \{\pm\omega^{(j)}\} }$ are finite length objects, they must be isomorphic.
\end{proof}

\section{Multitype Barsotti--Tate deformation rings}\label{sec:defring}

\subsection{\'Etale $\varphi$-modules} \label{subsec:phi}

Let $K_\infty$ be the infinite extension obtained by adjoining compatible $p$-power roots of $-p$ to $K$.
Let $\cO_{\cE,K}$ denote the $p$-adic completion of $\cO_K(\!(v)\!)$, and let $\cO_{\cE^{\mathrm{un}},K}$ denote the $p$-adic completion of a maximal connected \'etale extension of $\cO_{\cE,K}$.
For $R$ a complete local Noetherian $\cO$-algebra, let $\Phi\textrm{-}\mathrm{Mod}^{\et}(R)$ be the category of \'etale $\varphi$-modules over $\cO_{\cE,K}\otimes_{\Z_p} R$, and let $\mathrm{Rep}_{G_{K_{\infty}}}(R)$ be the category of (continuous) representations of $G_{K_{\infty}}$ over $R$.
Fontaine defined an exact anti-equivalence of tensor categories
\[\mathbb{V}^*: \Phi\textrm{-}\mathrm{Mod}^{\et}(R) \ra \mathrm{Rep}_{G_{K_{\infty}}}(R)\]
by $\mathbb{V}^*(\cM) = ((\cM\otimes \cO_{\cE^{\mathrm{un}},K})^{\varphi=1})^\vee$.

For a natural number $d$, let $\varpi_d \in E$ be a root of $u^{p^{df}-1}+p$.
Let $K_d$ be the degree $d$ unramified extension of $K$.
We define the fundamental character
\begin{align*}
\omega_{df}: G_{K_d} &\ra \cO^\times \\
g &\mapsto \frac{g(\varpi_d)}{\varpi_d},
\end{align*}
which does not depend on the choice of $\varpi_d$.
For $\alpha \in \F^\times$, denote by $\nr_\alpha$ the unramified character of $G_K$ taking a geometric Frobenius element to $\alpha$.

Let $\rhobar: G_K \ra \GL_2(\F)$ be a continuous Galois representation.
If $\rhobar$ is reducible, then it is an extension of 
\[
\nr_{\alpha'} \omega_f^{\sum_{i=0}^{f-1} \mu_{2,i}p^i} \textrm{ by } \nr_\alpha \omega_f^{\sum_{i=0}^{f-1} \mu_{1,i}p^i}
\]
for some dominant $p$-restricted character $\mu_\rhobar = (\mu_{1,i},\mu_{2,i})_i \in X^*(T)$ and some $\alpha$ and $\alpha' \in \F^\times$.
If $\rhobar$ is irreducible, then $\rhobar$ is 
\[
\Ind_{G_{K_2}}^{G_K} \nr_{-\alpha}\omega_{2f}^{\sum_{i=0}^{f-1} \mu_{1,i}p^i+p^f\sum_{i=0}^{f-1} \mu_{2,i}p^i}
\]
where $\mu_\rhobar$ again is a dominant $p$-restricted element of $X^*(T)$ and $\alpha \in \F^\times$.
We note that the main result of this paper in the case when $\rhobar$ is irreducible already appears in \cite{LMS} and \cite{HW}, and so this case can be ignored if the reader desires.
\cite{BDJ} attaches to $\rhobar$ a set $W(\rhobar)$ of Serre weights (see also \cite[\S 4, Proposition A.3]{breuil} with the notation $\mathcal{D}(\rhobar)$).

In both the reducible and irreducible cases, we now assume that $\mu_\rhobar\in X^*(T)$ with $\mu_i = (\mu_{1,i},\mu_{2,i}) = (c_i,1)$ with $3 < c_i < p-2$ for all $i\in \Z/f$.
For $i \in \Z/f$, let $a_i$ be an element of $\F$.
Let $\cM = \prod_i \F(\!(v)\!)\fE^i \oplus \F(\!(v)\!)\fF^i$ be the $\varphi$-module defined by
\begin{eqnarray*}
i \neq 0 : & &
\begin{cases}
\varphi(\fE^{i-1}) & = v^{c_{f-i}}\fE^i+a_{i-1}v^{c_{f-i}}\fF^i \\
\varphi(\fF^{i-1}) & = v\fF^i
\end{cases} \\
i=0, \, \rhobar \textrm{ reducible} : & &
\begin{cases}
\varphi(\fE^{f-1}) & = \alpha v^{c_0}\fE^0+\alpha a_{f-1}v^{c_0}\fF^0 \\
\varphi(\fF^{f-1}) & = \alpha' v\fF^0
\end{cases} \\
i=0, \, \rhobar \textrm{ irreducible} : & &
\begin{cases}
\varphi(\mathfrak{e}^{f-1}) & = \alpha v^{c_0}\fF^0  \\
\varphi(\mathfrak{f}^{f-1}) & = - v\fE^0
\end{cases} \\
\end{eqnarray*}
(here the $i$-th factor corresponds to the embedding $\iota_{-i}$).

\begin{prop}\label{prop:rhophi}
There are unique values $a_i \in \F$ for $i\in \Z/f$ such that $\mathbb{V}^*(\cM)$ is isomorphic to the restriction $\rhobar|_{G_{K_\infty}}$.
%The representation $\rhobar$ is semisimple if and only if $a_i$ is $0$ for all $i\in \Z/f$.
\end{prop}
\begin{proof}
Note that $\rhobar$ is Fontaine--Laffaille by the genericity condition.
We use Fontaine--Laffaille theory as in \cite[Appendix A]{breuil}.
We address the case when $\rhobar$ is reducible and leave the irreducible case to the reader.
Let $M = \oplus_{i\in \Z/f} M^{(i)}$ with $M^{(i)} = k_E e^{(i)} \oplus k_E f^{(i)}$ be the Fontaine--Laffaille module with
\[
\Fil^1 M^{(i)} = M^{(i)}, \, \Fil^2 M^{(i)} = \Fil^{c_{f-i}} M^{(i)} = k_E f^{(i)}, \, \Fil^{c_{f-i}+1} M^{(i)} = 0, \\
\]
\begin{align*}
\varphi(e^{(i)}) &=  e^{(i+1)}, \\
\varphi_{c_{f-i}}(f^{(i)}) &= f^{(i+1)}+a_{i-1} e^{(i+1)}, \, \textrm{for  } i\neq 1 \textrm{ and } \\
\varphi(e^{(1)}) &=  \alpha' e^{(2)}, \\
\varphi_{c_{f-1}}(f^{(1)}) &= \alpha f^{(2)}+\alpha' a_0 e^{(2)},
\end{align*}
for $a_i \in k_E$ such that $\rhobar \cong \Hom_{\Fil^\bullet,\varphi_.}(M,A_{\mathrm{cris}} \otimes_{\bZ_p} \bF_p)$ (see e.g.~\cite[(16)]{breuil}).
%In the paragraph following \cite[(16)]{breuil}, it is noted that $a_i = 0$ for all $i\in \Z/f$ if and only if $\rhobar$ is split reducible.

Let $\fM$ be the $\F_q[\![v]\!]\otimes_{\Z_p} \F$-submodule of $\cM$ generated by $(\fE_i)_{i\in \Z/f}$ and $(\fF_i)_{i\in \Z/f}$.
Note that $\varphi$ maps $\fM$ to itself.
Then a calculation (cf.~\cite[\S 7.4]{EGS} with $J = \emptyset$) shows that $\Theta_{p-1}(\fM) \cong \mathcal{F}_{p-1}(M)$, where the functors $\Theta_{p-1}$ and $\mathcal{F}_{p-1}$ are introduced in \cite[Appendix A]{EGS}.
The result now follows from \cite[Propositions A.3.2 and A.3.3]{EGS}.
\end{proof}

For the rest of this section, we fix, for each $i\in \Z/f$, $a_i\in\F$, the unique element as in Proposition \ref{prop:rhophi}.
In doing so, we thus fix $\cM$.
If $\rhobar$ is irreducible, let $S_\rhobar$ be the set $\{-\omega^{(0)}, \omega^{(1)}, \ldots ,\omega^{(f-1)}\}$.
Otherwise, let $S_\rhobar$ be the set $\{ \omega^{(i)} | a_{f-1-i} = 0\}$.

\begin{prop}\label{prop:serrewts}
The set $W(\rhobar)$ equals $\{ \sigma_J| J \subset S_\rhobar \}$ where $\sigma_J$ is defined with respect to $\mu_\rhobar$.
\end{prop}
\begin{proof}
This follows from a direct calculation using \cite[\S 4]{breuil}.
\end{proof}

\subsection{Kisin modules and deformation rings}

To describe tamely potentially Barsotti--Tate deformation rings, we will use the theory of Kisin modules with descent datum.
Let $\tau$ be the tame principal series type $\eta_1\oplus \eta_2:I_K \ra \GL_2(\F_q)$ where $\eta_k = \omega_f^{-\mathbf{a}_k^{(0)}}$ for $k=1$ and $2$ and 
\[\mathbf{a}_k^{(j)} = \sum_{i=0}^{f-1} a_{k,-j+i}p^i,\]
where $a_{k,i} \in \Z$.
We will suppose throughout that $2\leq |a_{1,i}-a_{2,i}| \leq p-3$ for all $i\in \Z/f$ and call such a tame principal series type {\it generic}.
We will say a tame inertial type $\tau'$ is generic if its restriction to the quadratic unramified extension of $K$ is a generic principal series type.

The {\it orientation} of $(\mathbf{a}_1,\mathbf{a}_2)$ is the element $s\in W$ such that $\mathbf{a}^{(j)}_{s_j(1)}>\mathbf{a}^{(j)}_{s_j(2)}$.
By an abuse of notation, we say that the orientation of $(\mathbf{a}_1,\mathbf{a}_2)$ is an orientation for $\tau$ if $\tau$ can be expressed in terms of $(\mathbf{a}_1,\mathbf{a}_2)$ as above.

Let $R$ be an $\cO$-algebra.
For a principal series type $\tau$, we will consider Kisin modules over $R$ with descent datum of type $\tau$ (see \cite[Definition 2.4]{LLLM}).
We will say that such a Kisin module $\fM_R$ is in $Y^{(0,1),\tau}(R)$ if the cokernels of $\phi_{\fM_R}:\varphi^*(\fM_R) \ra \fM_R$ and $\phi_{\det \fM_R}:\varphi^*(\det \fM_R) \ra \det \fM_R$ are annihilated by $E(u) = u^{q-1}+p$.
Let $v$ be $u^{q-1}$.

Let $s$ be an orientation for a generic tame principal series type $\tau$ and $\fM_R$ be an element of $Y^{(0,1),\tau}(R)$.
Then $\fM_R$ can be described by the matrices $\Mat_{\beta}(\phi^{(i)}_{\fM_R \otimes_R \F,s_{i+1}(2)})$ after choosing an eigenbasis $\beta$ (see \cite[Definition 2.11]{LLLM}).
The following is a generalization of \cite[Theorem 4.1]{LLLM} in the case of $\GL_2$, where $\beta$ is allowed to have a slightly more general form than a gauge basis.

\begin{thm}\label{thm:fhdef}
Let $\tau$ be a tame generic principal series type and let $s = (s_i)_i \in W$ be an orientation for $\tau$.
Let $R$ be a complete local Noetherian $\cO$-algebra with residue field $\F$.
Let $\fM_R\in Y^{(0,1),\tau}(R)$ with $\Mat_{\overline{\beta}}(\phi^{(i)}_{\fM_R \otimes_R \F,s_{i+1}(2)})$ given by 
\[
\overline{A}_1 = \begin{pmatrix}
  v &   \\
  a_iv & 1 \\
 \end{pmatrix},
\overline{A}_2 = \begin{pmatrix}
   & 1 \\
  v &  \\
 \end{pmatrix},
\overline{A}_3 = \begin{pmatrix}
   & 1  \\
  v & a_i \\
 \end{pmatrix},
\textrm{ or }
\overline{A}_4 = \begin{pmatrix}
  1 &   \\
   & v \\
 \end{pmatrix}
\]
for $i\neq 0$ and $\overline{A}_j\begin{pmatrix}
  \alpha &  \\
   & \alpha' \\
 \end{pmatrix}$ for $i=0$, where $\overline{\beta}$ is an eigenbasis for $\fM_R \otimes_R \F$.
Then there is a unique eigenbasis $\beta$ of $\fM_R$ up to scaling lifting $\overline{\beta}$ such that $\Mat_\beta(\phi^{(i)}_{\fM_R,s_{i+1}(2)})$ is given by 
\[
A_1 = \begin{pmatrix}
  v+p &   \\
  (X_i+[a_i])v & 1 \\
 \end{pmatrix},
A_2 = \begin{pmatrix}
  -Y_i & 1 \\
  v & X_i \\
 \end{pmatrix}, \]
\[
A_3 = \begin{pmatrix}
  -p(X_i+[a_i])^{-1} & 1  \\
  v & X_i+[a_i] \\
 \end{pmatrix},
\textrm{ or }
A_4 = \begin{pmatrix}
  1 & -Y_i \\
   & v+p \\
 \end{pmatrix}, \\
\]
respectively, for $i\neq 0$ and $A_jD(\alpha,\alpha')$ with $A_j$ as above for $i=0$.
Here $[\cdot]$ denotes the Teichm\"uller lift, $X_iY_i = p$ for $A_2$, and $D(\alpha,\alpha') = \begin{pmatrix}
  [\alpha]+ X_\alpha&  \\
   & [\alpha']+X_{\alpha'} \\
 \end{pmatrix}$.
\end{thm}
\begin{proof}
The proof is similar to the proof of \cite[Theorems 4.1 and 4.16]{LLLM} which prove existence and uniqueness of $\beta$, respectively.
We describe some of the key points.
We modify \cite[Definition 4.2]{LLLM}, defining $d_R(P) = \min_k 2v_R(r_k)+k$ if $P = \sum_k r_k v^k \in R[\![v]\!]$.
Then the analogue of \cite[Proposition 4.3]{LLLM} holds (see \cite[Remark 4.4]{LLLM}).
The entry in the middle column of \cite[Table 5]{LLLM} becomes 
\[
\begin{pmatrix}
  1^* &   \\
  v(\leq 0) & 0^* \\
 \end{pmatrix},
\begin{pmatrix}
  \leq 0 & 0^* \\
  1^* & \leq 0 \\
 \end{pmatrix},
\begin{pmatrix}
  \leq 0 & 0^* \\
  1^* & \leq 0 \\
 \end{pmatrix},
\textrm{ or }
\begin{pmatrix}
   0^* & \leq 0 \\
  &  1^* \\
 \end{pmatrix},\]
respectively, and we modify \cite[Definition 4.5]{LLLM} for $E^{(i)}$ appropriately.
For $1\leq m,k \leq 2$, we define $\delta(A^{(i)}_{mk})$ to be $d_R(E^{(i)}_{mk})$ if $\overline{A}^{(i)} \neq \overline{A}_3$.
If $\overline{A}^{(i)} = \overline{A}_3$, we define $\delta(A^{(i)}_{mk})$ to be $d_R(E^{(i)}_{mk})$ (resp.~$d_R(E^{(i)}_{mk})+1$) if $k=1$ (resp.~if $k=2$).
Finally, we let
\[\delta(A^{(i)}) = \min_{1\leq m,k\leq 2}\{\delta(A_{mk}^{(i)})\}.\]
The analogue of \cite[Proposition 4.6]{LLLM} holds, replacing $3+ d_R(x^{(j)})$ with $2+d_R(x^{(j)})$.
We define the notion of pivots for $\overline{A}^{(i)} \neq \overline{A}_3$ as in the \cite[Definition 4.8]{LLLM}, and define the pivots in the case of $\overline{A}^{(i)} = \overline{A}_3$ to be the same as the pivots in the case of $\overline{A}_2$.
The analogue of \cite[Lemma 4.10]{LLLM} holds except that the second equation of \emph{loc.~cit.}~is changed to $A_{22}^{(i)} = vP_{22} + [a_i] + Q_{22}$ when $\overline{A}^{(i)} = \overline{A}_3$.
Then the analogues of \cite[Proposition 4.11, Proposition 4.13, and Lemma 4.14]{LLLM} give the eigenbasis $\beta$.

We give more details for the algorithm in the case $\overline{A}^{(i)} = \overline{A}_3$. 
We let $\delta>1$ be an integer.
Suppose that $\delta(A^{(i)})$, which is necessarily greater than one, is $\delta$.
Then there is an $x \in R[\![v]\!]$ with $d_R(x) \geq \delta - 1$ such that ${A}^{\prime,(i)} \defeq D_{22}(x)A^{(i)}$ satisfies $\delta({A}^{\prime,(i)}) \geq \delta$ and $\delta({A}^{\prime,(i)}_{21}) > \delta$. 
Note the crucial role played by the definition of $\delta({A}^{\prime,(i)}_{22})$ as $d_R(E^{\prime,(i)}_{22})+1$ in this case.
Moreover, these inequalities still hold after right multiplication by a conjugate of $D_{22}(x)^\varphi$ by a permutation matrix. 
This is the analogue of \cite[Proposition 4.6]{LLLM}, where the notation $I^\varphi$ is defined.

Suppose next that $\delta(A^{(i)})$ is $\delta$ and that $\delta(A^{(i)}_{21}) > \delta$.
Then there exists an $x \in R[\![v]\!]$ with $d_R(x) \geq \delta - 1$ such that ${A}^{\prime,(i)} \defeq U_{12}(x)A^{(i)}$ satisfies $\delta({A}^{\prime,(i)}) \geq \delta$ and $\delta({A}^{\prime,(i)}_{11}),\, \delta({A}^{\prime,(i)}_{21}) > \delta$ (note that $\delta({A}^{\prime,(i)}_{21}) = \delta({A}^{(i)}_{21})$).
Again, we use that $\delta({A}^{\prime,(i)}_{12}) = d_R({E}^{\prime,(i)}_{12}) + 1$.
Moreover, these inequalities still hold after right multiplication by a conjugate of $U_{12}(x)^\varphi$ by a permutation matrix by the genericity assumption.

Suppose next that $\delta(A^{(i)})$ is $\delta$ and that $\delta(A^{(i)}_{11}),\, \delta(A^{(i)}_{21}) > \delta$.
Then there is an $x \in R[\![v]\!]$ with $d_R(x) \geq \delta - 1$ such that ${A}^{\prime,(i)} \defeq D_{11}(x)A^{(i)}$ satisfies $\delta({A}^{\prime,(i)}) \geq \delta$ and $\delta({A}^{\prime,(i)}_{11}),\, \delta({A}^{\prime,(i)}_{21}).\, \delta({A}^{\prime,(i)}_{12}) > \delta$ using that $A^{(i)}_{11} \in m_R \cdot R[\![v]\!]$. 
Moreover, these inequalities still hold after right multiplication by a conjugate of $D_{11}(x)^\varphi$ by a permutation matrix.

Suppose finally that $\delta(A^{(i)})$ is $\delta$ and that $\delta(A^{(i)}_{11}),\, \delta(A^{(i)}_{21}),\, \delta({A}^{(i)}_{12}) > \delta$.
Then there is an $x \in R[\![v]\!]$ with $d_R(x) \geq \delta - 1$ such that ${A}^{\prime,(i)} \defeq L_{21}(x)A^{(i)}$ satisfies $\delta({A}^{\prime,(i)}) \geq \delta+1$ using again that $A^{(i)}_{11} \in m_R \cdot R[\![v]\!]$. 
Moreover, these inequalities still hold after right multiplication by a conjugate of $L_{21}(x)^\varphi$ by a permutation matrix by the genericity assumption.
Repeating these four steps repeatedly gives the analogue of \cite[Proposition 4.13]{LLLM} in this case.

We deduce the forms of $A_i$ from the condition that $v+p$ must divide the determinant.
Finally, the analogue of \cite[Theorem 4.16]{LLLM} proves the uniqueness of $\beta$ up to scaling.
In the notation of \emph{loc.~cit.}, we obtain the equation 
\begin{equation}\label{eqn:changebasis}
\tld{A}_2^{(i)}+v^2\tld{A}_2^{(i)}M^{(i)} = \tld{A}_1^{(i)}+I^{(i+1)} \tld{A}_1^{(i)}
\end{equation}
(cf.~\cite[(4.2)]{LLLM}).
Suppose that $d_R(I^{(j)}) \geq \delta \geq 1$ for all $j$.
Then one can show that $d_R(I^{(j)}) \geq \delta+1$ for all $j$.
This implies that $I^{(j)} = 0$ for all $j$.
We again give more details in the case $\overline{A}^{(i)} = \overline{A}_2$ or $\overline{A}_3$.
The other cases are treated similarly.
Let $k$ be $1$ or $2$.
We first compare the $(k,1)$-entries of (\ref{eqn:changebasis}) to see that $d_R(I^{(i+1)}_{k2}) \geq \delta+1$.
Using this and the $(k,2)$-entries of (\ref{eqn:changebasis}), we see that $d_R(I^{(i+1)}_{k1}) \geq \delta+1$.
\end{proof}

%For a Galois representation $\rhobar:G_K \ra \GL_2(\F)$ and a tame inertial type $\tau$, let $R^\tau$ be the universal $\cO$-algebra parameterizing liftings of $\rhobar$ which are potentially Barsotti--Tate of type $\tau$.

For the rest of the section, let $\rhobar$ be as in \S \ref{subsec:phi}, and let $\cM$ be as in Proposition \ref{prop:rhophi} so that $\rhobar|_{G_{K_\infty}}$ is isomorphic to $\mathbb{V}^*(\cM)$.
Moreover, for simplicity, assume that $\rhobar$ is reducible.
Recall the definition of $S_\rhobar$ from \S \ref{subsec:phi}.

Let $s$ and $s'$ be in $W$ such that one of the following holds for each $i\in\Z/f$:
\begin{enumerate}
\item \label{item:Y} $s_i$ and $s_i'$ are both $\id$;
\item \label{item:XY} $s_i$ and $s_i'$ are both not $\id$; and
\item \label{item:X} $s_i$ is $\id$, but $s_i'$ is not, and $i\in S_\rhobar$.
\end{enumerate}
We say that $i\in \Z/f$ is case (\ref{item:Y}), (\ref{item:XY}), or (\ref{item:X}) if the above relevant condition holds.

\begin{prop}\label{prop:1defring}
Let $s$ and $s'$ be in $W$ as above.
Let $\tau$ be the tame generic inertial type with $\sigma(\tau) \cong R_s(\mu_\rhobar-s'\eta)$.
Let $R$ be the ring $\cO[\![ (X_i,Y_i)_{i=0}^{f-1}, X_\alpha,X_{\alpha'}]\!]/(h_i)$ where for each $i\in \Z/f$, $h_i$ is $Y_i$, $X_iY_i-p$, $Y_i-p$, or $X_i$ if $f-1-i$ is case $($\ref{item:Y}$)$, $($\ref{item:XY}$)$ with $\omega^{(f-i)} \in S_\rhobar$, $($\ref{item:XY}$)$ with $\omega^{(f-i)} \notin S_\rhobar$, or $($\ref{item:X}$)$, respectively.
Let $\cM_R = \prod_i R(\!(v)\!)\fE^i \oplus R(\!(v)\!)\fF^i$ be the $\varphi$-module defined by 
\begin{eqnarray*}
f-i \textrm{ is case } (\ref{item:Y}): & &
\begin{cases}
\varphi(\fE^{i-1}) & = v^{c_{f-i}-1}(v+p)\fE^i+(X_{i-1}+[a_{i-1}])v^{c_{f-i}}\fF^i \\
\varphi(\fF^{i-1}) & = v\fF^i
\end{cases} \\
f-i \textrm{ is case } (\ref{item:XY}), \omega^{(f-i)} \in S_\rhobar : & &
\begin{cases}
\varphi(\mathfrak{e}^{i-1}) & = v^{c_{f-i}}\mathfrak{e}^i + X_{i-1}v^{c_{f-i}}\fF^i \\
\varphi(\mathfrak{f}^{i-1}) & = -Y_{i-1}\fE^i + v \fF^i
\end{cases} \\
f-i \textrm{ is case } (\ref{item:XY}), \omega^{(f-i)} \notin S_\rhobar : & &
\begin{cases}
\varphi(\mathfrak{e}^{i-1}) & = v^{c_{f-i}}\mathfrak{e}^i + (X_{i-1}+[a_{i-1}])v^{c_{f-i}}\fF^i \\
\varphi(\mathfrak{f}^{i-1}) & = -p(X_{i-1}+[a_{i-1}])^{-1}\fE^i+v\mathfrak{f}^i
\end{cases} \\
f-i \textrm{ is case } (\ref{item:X}) : & &
\begin{cases}
\varphi(\mathfrak{e}^{i-1}) & = v^{c_{f-i}}\fE^i \\
\varphi(\mathfrak{f}^{i-1}) & = -Y_{i-1}\fE^i +  (v+p)\fF^i,
\end{cases} \\
\end{eqnarray*}
with the usual modification for $i=0$.
Then $\mathbb{V}^*(\cM_R)$ is the restriction to $G_{K_\infty}$ of a versal potentially Barsotti--Tate deformation of $\rhobar$ of type $\tau$.
\end{prop}
\begin{proof}
Define $w^*\in W$ and $s_\tau \in S_2$ to be the unique elements such that $w^*_{f-1} = \Id$ and $(w^*)^{-1} s \pi(w^*) = (s_\tau,\Id,\ldots, \Id)$.
Then the Deligne--Lusztig representations $R_s(\mu_\rhobar-s'\eta)$ and $R_{(s_\tau,\Id,\ldots, \Id)}((w^*)^{-1}(\mu_\rhobar-s'\eta))$ are isomorphic by \cite[Lemma 4.2]{Herzig}.
Moreover, (the quadratic base change of) $R_{(s_\tau,\Id,\ldots, \Id)}((w^*)^{-1}(\mu_\rhobar-s'\eta))$ is a generic principal series.
Define $w = (w_i)_i$ by $w_i = (w^*_{f-1-i})^{-1}$ for $i\in \Z/f$.
Then one easily checks that $w$ is an orientation for $(w^*)^{-1}(\mu_\rhobar-s'\eta)$.
Let $\fM_R$ be the Kisin module (with quadratic unramified descent) of tame inertial type (the quadratic unramified base change of) $\tau(s_\tau,-(w^*)^{-1}(\mu_\rhobar-s'\eta))$ with $A^{(i-1)} = \Mat_\beta(\phi^{(i-1)}_{\fM_R,w_i(2)})$ given by $A_1$, $A_2$, $A_3$, or $A_4$ if $f-i$ is case (\ref{item:Y}), $f-i$ is case (\ref{item:XY}) and $f-i \in S_\rhobar$, $f-i$ is case (\ref{item:XY}) and $f-i \notin S_\rhobar$, or $f-i$ is case (\ref{item:X}), respectively.
We claim that $T_{dd}^*(\fM_R\otimes_{\cO} \F)$ is isomorphic to the restriction to $G_{K_\infty}$ of $\rhobar$.
Assuming this, by Theorem \ref{thm:fhdef} and the analogue of \cite[\S 5.2 and \S 6]{LLLM}, $T_{dd}^*(\fM_R)$ is the restriction to $G_{K_\infty}$ of a versal potentially Barsotti--Tate deformation of $\rhobar$ of type $\tau$.

Let $L$ be $K((-p)^{\frac{1}{e}})$ with $e = q-1$ if $s_\tau = \id$ and $K_2((-p)^{\frac{1}{e}})$ with $e = q^2-1$ otherwise.
Let $\Delta$ be the Galois group $\Gal(L/K)$.
We claim that 
\[
(\fM_R \otimes_{\cO_{\cE,K}} \cO_{\cE,L})^\Delta \cong \cM_R.
\]
This would finish the proof including the claim in the previous paragraph since the restriction to $G_{K_\infty}$ of $\rhobar$ is isomorphic to $\cM$ by Proposition \ref{prop:rhophi}, and clearly $\cM_R \otimes_{\cO} \F$ is isomorphic to $\cM$.

Let $\mu_\rhobar$ be $(\mu_i)_i$.
Let $v^\lambda$ denote the torus element obtained by applying the coweight $\lambda$ to $v \defeq u^e$.
By \cite[Proposition 3.1.2]{LLLM2}, we see that a Kisin module (with quadratic unramified descent) of tame inertial type (the quadratic unramified base change of) $\tau$ with $\Mat_\beta(\phi^{(i)}_{\fM,w_{i+1}(2)})$ given by $A^{(i)}$ (resp. $A^{(i)}s_0^{-1} D(\alpha,\alpha')s_0$) for $i< f-1$ (resp. for $i=f-1$) gives a $\varphi$-module $\cM = \prod_i \F(\!(v)\!){\fE'}^i \oplus \F(\!(v)\!){\fF'}^i$ with $\varphi({\fE'}^{i-1},{\fF'}^{i-1}) = M_{i-1}'({\fE'}^i,{\fF'}^i)$ where 
\begin{align*}
M_i' &= w_{i+1} A^{(i)} v^{w_{i+1}^{-1}(w^*_{f-1-i})^{-1}(\mu_{f-1-i} - s'_{f-1-i}\eta)}(w_{i+1})^{-1} \\
&= (w^*_{f-2-i})^{-1}  A^{(i)} v^{w^*_{f-2-i}(w^*_{f-1-i})^{-1}(\mu_{f-1-i} - s'_{f-1-i}\eta)}w^*_{f-2-i} \\
&= (w^*_{f-2-i})^{-1}  A^{(i)} v^{s_{f-1-i}^{-1}(\mu_{f-1-i} - s'_{f-1-i}\eta)}w^*_{f-2-i}
\end{align*}
for $i< f-1$ and $M_{f-1}' = A^{(f-1)}s_0^{-1} D(\alpha,\alpha')s_0s_\tau^{-1}v^{(w^*_0)^{-1}(\mu_0 - s_0\eta)}$.
Changing to the bases $(\fE^i,\fF^i) = ({\fE'}^i,{\fF'}^i)(w^*_{f-2-i})^{-1}$, we see that $\cM$ is given by $(M_i)_i$ where 
\begin{align*}
M_i &= A^{(i)} v^{s_{f-1-i}^{-1}(\mu_{f-1-i} - s_{f-1-i}'\eta)}w_{f-2-i}^* (w_{f-1-i}^*)^{-1} \\
&= A^{(i)} v^{s_{f-1-i}^{-1}(\mu_{f-1-i} - s_{f-1-i}'\eta)}s_{f-1-i}^{-1} \\
&= A^{(i)} s_{f-1-i}^{-1} v^{\mu_{f-1-i} - s_{f-1-i}'\eta}
\end{align*}
for $i < f-1$ and 
\begin{align*}
M'_{f-1} &= A^{(f-1)}s_0^{-1} D(\alpha,\alpha')s_0s_\tau^{-1}v^{(w^*_0)^{-1}(\mu_0 - s_0'\eta)}(w^*_0)^{-1} \\
&= A^{(f-1)}s_0^{-1} D(\alpha,\alpha')s_0s_\tau^{-1}(w^*_0)^{-1} v^{\mu_0 - s_0'\eta} \\
&= A^{(f-1)}s_0^{-1} v^{\mu_0 - s_0'\eta}D(\alpha,\alpha').
\end{align*}
The proposition is now deduced by substituting for $A^{(i)}$, $s$, and $\mu_\rhobar$.
\end{proof}

%Let $\rhobar:G_K \ra \GL_2(\F)$ is a Galois representation.
If $\tau$ is an inertial type, let $R^\tau$ parameterize potentially Barsotti--Tate (framed) liftings of $\rhobar$ of type $\tau$.
If $T$ is a set of inertial types for $K$, then we let $\Spec R^T$ be the Zariski closure of $\cup_{\tau \in T} \Spec R^\tau[p^{-1}]$ in the universal (framed) lifting space $\Spec R_{\rhobar}^\square$ of $\rhobar$.

For applications to Shimura curves and algebraic modular forms on definite quaternion algebras, it is convenient to consider fixed determinant deformation rings.
If $\psi: G_K \ra \cO^\times$ is a continuous character, let $R^{\psi,\square}_\rhobar$ be the quotient of $R^\square_\rhobar$ parameterizing (framed) liftings of $\rhobar$ with determinant $\psi \varepsilon$.
Let $R^{\psi,\tau}$ be the simultaneous quotient of $R^{\psi,\square}_\rhobar$ and $R^\tau$ parameterizing potentially Barsotti--Tate (framed) liftings of $\rhobar$ of type $\tau$ and determinant $\psi \varepsilon$.
We can similarly define the quotient $R^{\psi,T}$ of $R^T$.
If $R^{\psi,\tau}$ is nonzero, then $R^\tau$ must be nonzero, $\psi$ must lift $\ovl{\varepsilon}^{-1} \det \rhobar$, and $\psi|_{I_K}$ must be $\det \tau$.
For all sets of types $T$ considered below, the determinants of all elements of $T$ coincide.

Now fix a Serre weight $\sigma$ in $W(\rhobar)$.
Suppose that $\sigma = \sigma_J$ for $J \subset S_\rhobar$ where $\sigma_J$ is defined with respect to $\mu_\rhobar$.
Let $I$ be a subset of $S$ such that $I \cap \{ \pm \omega^{(i)}\}$ has size at most one for all $i\in \Z/f$.
Let $T_{J,I}$ be the set of inertial types $\tau$ such that $\sigma(\tau)$ is of the form $R_s(\mu_\rhobar-s'\eta)$ where $s$ and $s'$ have the restrictions given by the following table.
\begin{center}
\begin{tabular}{ |c||c|c| } 
\hline
$s_i,\, s'_i$ & $i \notin J$ & $i\in J$ \\
\hline
$\{ \pm\omega^{(i)}\} \cap I = \emptyset$ & $s_i = s'_i$ & $s' \neq \id$ \\
\hline
$\omega^{(i)} \in I$ & $s_i = s'_i = \id$ & $s_i = s_i' \neq \id$ \\
\hline
$-\omega^{(i)} \in I$ & $s_i = s'_i \neq \id$ & $s_i = \id, \, s'_i \neq \id$ \\
\hline
\end{tabular}
\end{center}

\begin{lemma}\label{lemma:illset}
Define $w_J\in W$ by $w_{J,i-1} = \id$ if and only if $i\notin J$ for all $i\in \Z/f$. 
Then the set of tame inertial types $T_{J,I}$ corresponds by inertial local Langlands to the set $T_{\sigma,w_J(I)}$ of Deligne--Lusztig representations defined in \S \ref{sec:proj}.
\end{lemma}
\begin{proof}
This is a computation using the definitions and \cite[Theorem 5.2]{Herzig}.
Note that in the notation of \emph{loc.~cit.}, $\gamma'_{\sigma,\tau}$ in this case is equal to the Kronecker symbol for $\sigma$ and $\tau$.
Another method of proof is to use \cite[Proposition 2.10]{LMS} and verify that if $V_\phi(\tau) \cong R_s(\mu)$, then $W^?(\tau) = \JH(\ovl{R}_{sw_0}(\mu-sw_0\eta))$.
\end{proof}

\begin{thm}\label{thm:defring}
There is an isomorphism from $R^{T_{J,I}}$ to a formal power series ring over $\cO[\![ (X_i,Y_i)_{i=0}^{f-1}]\!] /(g_i(J,I))_i$, where $g_i(J,I)$ is given by the following table.
\begin{center}
\begin{tabular}{ |c||c|c|c| } 
\hline
$g_i(J,I)$ & $\omega^{(f-1-i)} \notin S_\rhobar$ & $\omega^{(f-1-i)}\in S_\rhobar \setminus J$ & $\omega^{(f-1-i)}\in J$ \\
\hline
$\{ \pm\omega^{(f-1-i)}\} \cap I = \emptyset$ & $Y_i(Y_i-p)$ & $Y_i(X_iY_i-p)$ & $X_i(X_iY_i-p)$ \\
\hline
$\omega^{(f-1-i)} \in I$ & $Y_i$ & $Y_i$ & $X_iY_i-p$\\
\hline
$-\omega^{(f-1-i)} \in I$ & $Y_i-p$ & $X_iY_i-p$ & $X_i$ \\
\hline
\end{tabular}
\end{center}
If $I \subset I'$, then $g_i(J,I')|g_i(J,I)$ for all $i\in \Z/f$ and $R^{T_{J,I'}}$ is the quotient of $R^{T_{J,I}}$ by the ideal $(g_i(J,I'))_i$.
Analogous results hold for $R^{\psi,T_{J,I}}$ provided that $\psi$ is chosen so that $R^{\psi,T_{J,I}}$ is nonzero for any, or equivalently all, choices of $I$ as above.
\end{thm}
\begin{rmk}
Since twisting by the universal unramified deformation of the trivial character gives an isomorphism $R^T \cong R^{\psi,T}[\![X]\!]$ (assuming $R^{\psi,T}$ is nonzero), the fixed determinant case follows from the first part of Theorem \ref{thm:defring}, and we ignore it below (cf.~\cite[Remark 7.2.2]{EGS}).
\end{rmk}
\begin{proof}
Since $R^{T_{J,I}}$ is naturally a quotient of $R_{\rhobar|_{G_{K_\infty}}}^\square$ by \cite[Lemma 7.4.3]{EGS},
it suffices to compute the Zariski closure of $\cup_{\tau \in T_{J,I}} \Spec R^\tau[p^{-1}]$ in $\Spec R_{\rhobar|_{G_{K_\infty}}}^\square$.
Let $R$ be the ring $\cO[\![ (X_i,Y_i)_{i=0}^{f-1}, X_\alpha,X_{\alpha'}]\!] /(g_i(J,\emptyset))_i$ and consider the deformation $\cM_R = \prod_i R(\!(v)\!)\fE^i \oplus R(\!(v)\!)\fF^i$ of $\cM$ defined by 
\begin{eqnarray*}
f-i \notin S_\rhobar : & &
\begin{cases}
\varphi(\fE^{i-1}) & = v^{c_{f-i}-1}(v+p-Y_{i-1})\fE^i+v^{c_{f-i}}(X_{i-1}+[a_{i-1}])\fF^i \\
\varphi(\fF^{i-1}) & = -Y_{i-1}(X_{i-1}+[a_{i-1}])^{-1}\fE^i+v\fF^i
\end{cases} \\
f-i \in S_\rhobar \setminus J : & &
\begin{cases}
\varphi(\mathfrak{e}^{i-1}) & = v^{c_{f-i}-1}(v+p-X_{i-1}Y_{i-1})\mathfrak{e}^i + X_{i-1}v^{c_{f-i}}\fF^i \\
\varphi(\mathfrak{f}^{i-1}) & = -Y_{i-1}\fE^i + v\mathfrak{f}^i
\end{cases} \\
f-i \in J : & &
\begin{cases}
\varphi(\mathfrak{e}^{i-1}) & = v^{c_{f-i}}\mathfrak{e}^i + X_{i-1}v^{c_{f-i}}\fF^i \\
\varphi(\mathfrak{f}^{i-1}) & = -Y_{i-1}\fE^i + (v+p-X_{i-1}Y_{i-1})\mathfrak{f}^i
\end{cases} \\
\end{eqnarray*}
with the usual modification at $i=0$.
Define the deformation functor $D^\square$ by $D^\square(A) = \{(\psi:R\ra A,b_A)\}/\isom$ for $A$ a complete local Noetherian $\cO$-algebra, where $b_A$ is a basis for the free rank two $A$-module $\mathbb{V}^*(\psi^*(\cM_R))$ whose reduction modulo $\fm_A$ gives $\rhobar$.
Then the natural map $D^\square \ra \Spf R$ is a $\widehat{\GL}_2$-torsor and is thus formally smooth of dimension $4$.
Let $D^\square$ be $\Spf R^\square$.
One can rescale $\fE^0$ and $\fF^0$ by units, and rescale the other basis vectors appropriately so that the coefficients in the definition of $\varphi$ which are $1$ remain $1$.
This gives a $\widehat{\mathbf{G}}_m^2$-action on $R$, and orbits give isomorphic $\varphi$-modules.
We claim that the natural map $\Spf R^\square/\widehat{\mathbf{G}}_m^2 \ra \Spf R_{\rhobar|_{G_{K_\infty}}}^\square$ is a closed embedding.
It suffices to show injectivity on reduced tangent spaces.

Suppose that $t$ is a reduced tangent vector of $\Spf R^\square/\widehat{\mathbf{G}}_m^2$ which maps to zero in $\Spf R_{\rhobar|_{G_{K_\infty}}}^\square$.
By formal smoothness, we can extend this to a map $t: R^\square \ra \F[\varepsilon]/(\varepsilon^2)$.
Let $\cM_t$ be $\cM_R \otimes_{R,t} \F[\varepsilon]/(\varepsilon^2)$ so that $\cM_t$ and $\cM \otimes_{\F} \F[\varepsilon]/(\varepsilon^2)$ are isomorphic.
Let $M_i$ (resp.~ $M_{t,i}$) be the matrices such that $\varphi(\fE^i\otimes_R \F,\fF^i\otimes_R \F) = M_i(\fE^{i+1}\otimes_R \F,\fF^{i+1}\otimes_R \F)$ (resp.~ $\varphi(\fE^i\otimes_R \F[\varepsilon]/(\varepsilon^2),\fF^i\otimes_R \F[\varepsilon]/(\varepsilon^2)) = M_{t,i}(\fE^{i+1}\otimes_R \F[\varepsilon]/(\varepsilon^2),\fF^{i+1}\otimes_R \F[\varepsilon]/(\varepsilon^2))$).
Then there are matrices $D_i \in \GL_2(\F(\!(v)\!) )$ such that 
\[(\Id_3+\varepsilon D_i)M_i \varphi(\Id_3-\varepsilon D_{i-1}) = M_{t,i}\]
for all $i \in \Z/f$, where $\id_3$ is the $3 \times 3$ identity matrix (we can assume without loss of generality that the terms without $\varepsilon$ are $\Id_3$ by multiplying by their inverses).
We first claim that $D_i\in \GL_2(\F[\![v]\!])$ for all $i \in \Z/f$.
For each $i$, let $k_i \in \Z$ be the minimal integer such that $v^{k_i}D_i \in \Mat_3(\F[\![v]\!])$.
Then $v^{c_{f-1-i}+k_i}\varphi(\Id_3-\varepsilon D_{i-1}) = v^{c_{f-1-i}+k_i}M_i^{-1} (\Id_3-\varepsilon D_i) M_{t,i}\in \Mat_3(\F[\![v]\!])$, and thus $c_{f-1-i}+k_i\geq pk_{i-1}$.
Since $c_{f-1-i} < p-1$, $k_i\geq 2 + p(k_{i-1}-1)$.
If $k_{i-1} \geq n\geq 1$, then $k_i\geq n+1$, from which we derive the contradiction that $k_i\geq n$ for every $n\in \N$.
Hence $k_i \leq 0$ for all $i$.

We next claim that if $f-1-i\notin S_\rhobar$ for some $i\in \Z/f$, then $t(Y_i) = 0$.
Suppose for the sake of contradiction that $f-1-i\notin S_\rhobar$ and $t(Y_i) \neq 0$.
Let $N_i \in \Mat_3(\F[\![v]\!])$ be such that $\varepsilon N_i = M_{t,i} - M_i$.
Then by the formulas for $M_i$ and $M_{t,i}$, the first (resp.~ second) entry in the top row of $N_i$ is exactly divisible by $v^{c_{f-1-i}-1}$ (resp.~ $v^0$).
On the other hand, since $D_i M_i - M_i \varphi(D_{i-1}) = N_i$, the first (resp.~ second) entry in the top row of $N_i$ is divisible by $v^{c_{f-1-i}}$ (resp.~ $v$), which is a contradiction.
Thus $t$ is a reduced tangent vector of 
\[(\Spf R^\square/(Y_i:\, f-1-i\notin S_\rhobar))/\widehat{\mathbf{G}}_m^2.\]

Let $\tau$ be the tame intertial type such that $\sigma(\tau) = R_{w_0}(\mu-w_0\eta)$.
Then the natural map from the quotient of 
\begin{equation}\label{eqn:Rsquare}
\Spf R^\square/(\varpi,\{Y_i:\, f-1-i \notin S_\rhobar\},\, \{X_iY_i: f-1-i\in S_\rhobar\})
\end{equation}
by $\widehat{\mathbf{G}}_m^2$ to $\Spf R^\tau/\varpi$ is formally smooth by Proposition \ref{prop:1defring}.
In fact, it is an isomorphism since the domain and codomain are both of dimension $f+4$ over $\F$.
Indeed, for the codomain this follows from \cite[Theorem 3.3.4]{kisin} and $p$-flatness, while for the domain we see directly that (\ref{eqn:Rsquare}) has dimension $f+6$.
Since the map 
\begin{align*}
\Spf R^\square/(\varpi,\{Y_i:\, f-1-i \notin S_\rhobar\},\{X_iY_i: f-1-i\in S_\rhobar\}) \\
\ra \Spf R^\square/(\varpi,\, \{Y_i:\, f-1-i\notin S_\rhobar\})
\end{align*}
is an isomorphism on reduced tangent spaces, $t$ is a reduced tangent vector of $\Spf R^\tau$.
Since $\Spf R^\tau \ra \Spf R_{\rhobar|_{G_{K_\infty}}}^\square$ is injective on reduced tangent spaces again by \cite[Lemma 7.4.3]{EGS}, $t$ is zero.

Finally, since $R$ is $p$-flat, it suffices to show that if $\# (\{\pm\omega_i\} \cap I) = 1$ for all $i\in \Z/f$, then $\mathbb{V}^*(\cM/(g_i(J,I))_i)$ is the restriction to $G_{K_\infty}$ of a versal potentially Barsotti--Tate deformation of $\rhobar$ of the unique type $\tau$ in $T_{J,I}$.
This follows from Proposition \ref{prop:1defring}.
\end{proof}

\section{Patching functors and multiplicity one}

Let $\rhobar: G_K \ra \GL_2(\F)$ be a continuous Galois representation.
Again, $\rhobar$ is either an extension of
\[
\nr_{\alpha'} \omega_f^{\sum_{i=0}^{f-1} \mu_{2,i}p^i} \textrm{ by } \nr_\alpha \omega_f^{\sum_{i=0}^{f-1} \mu_{1,i}p^i}
\]
or is
\[
\nr_\alpha \Ind_{G_{K_2}}^{G_K} \omega_{2f}^{\sum_{i=0}^{f-1} \mu_{1,i}p^i+p^f\sum_{i=0}^{f-1} \mu_{2,i}p^i}
\]
for some dominant $p$-restricted character $\mu_\rhobar = (\mu_{1,i},\mu_{2,i})_i \in X^*(T)$ and some $\alpha$ and $\alpha' \in \F^\times$.

\begin{defn} \label{def:gen}
We say that a dominant $p$-restricted $\mu \in X^*(T)$ is generic if $2 < \langle \mu, \beta \rangle < p-3$. 
We say that $\rhobar$ is generic if $\mu_\rhobar$ is generic or if $\rhobar$ is semisimple and $1$-generic in the sense of \cite[Definition 4.1]{LMS}.
\end{defn}

Note that if $\rhobar$ is generic, then $\rhobar$ is generic in the sense of \cite[Definition 11.7]{BP} and \cite[Definition 2.1.1]{EGS}.
We now assume that $\rhobar$ is not semisimple and is generic.
Then a twist of $\rhobar$ is of the form in \S \ref{subsec:phi}.

We now fix a Serre weight $\sigma \in W(\rhobar)$ ($W(\rhobar)$ is recalled in \S \ref{subsec:phi}).
Let $\mu\in X^*(T)$ be such that $\sigma \cong F(\mu-\eta)$.
If $\sigma$ is $\sigma_{J(\sigma)}$ with respect to $\mu_\rhobar$, define $w_{J(\sigma)}\in W$ by $w_{J(\sigma),i-1} = \id$ if and only if $i\notin J(\sigma)$ for all $i\in \Z/f$ as in Lemma \ref{lemma:illset}.
Then we set $S_\rhobar^\sigma$ to be $w(S_\rhobar)$ with $w = w_{J(\sigma)}^{-1} \pi(w_{J(\sigma)})$.

\begin{lemma}
The set $W(\rhobar)$ is $\{ \sigma_J|J\subset S_\rhobar^\sigma\}$ where $\sigma_J$ is defined in terms of $\mu$.
\end{lemma}
\begin{proof}
This follows from Proposition \ref{prop:serrewts} and \cite[Proposition 2.4]{LMS}.
\end{proof}

Let $\psi:G_K \ra \cO^\times$ be an unramified twist of $\omega_f^{\sum_{i\in \Z/f} (\mu_{1,i}+\mu_{2,i}-1)}$ lifting $\ovl{\varepsilon}^{-1}\det \rhobar$.
Suppose that $M_\infty(\cdot)$ is a minimal fixed determinant patching functor over $\cO$ for $\rhobar^\vee$ with fixed determinant $\psi^\vee$ (see \cite[Definition 6.1.3]{EGS}).
(Note that $\mathcal{D}(\rhobar^\vee)$ in the conventions of \cite[\S 2]{EGS} is $W(\rhobar)$ in ours.)
Using contragredients, we identify $R_{\rhobar^\vee}^\square$ with $R_\rhobar^\square$.
This identifies $R^\tau$ with the (framed) lifting ring of $\rhobar^\vee$ parameterizing lifts $\rho^\vee$ of type $\tau^\vee$ with $\HT_\kappa(\rho^\vee) = \{-1,0\}$ for all $\kappa: E\into \C_p$.
Note that such lifts of $\rhobar^\vee$ are called potentially Barsotti--Tate in \cite[\S 7]{EGS}.
Similar identifications are made for multitype (fixed determinant) potentially Barsotti--Tate deformation rings.
For an $\cO_K[\GL_2(\cO_K)]$-module $N$, we will denote $M_\infty(N\otimes_{\cO_K} \cO)$ by $M'_\infty(N)$, where tensor product is over the map $\cO_K\into \cO$ in \S \ref{subsec:not}.

\begin{lemma}\label{lemma:patch2}
The $R_\infty$-module $M'_\infty(R_{\mu}/\Fil^2_\otimes R_{\mu})$ is cyclic.
\end{lemma}
\begin{proof}
Let $\tau$ be the tame type such that $\sigma(\tau) = R_w(\mu-w\eta)$.
Then $W(\rhobar)$ is exactly $\JH(\ovl{\sigma}(\tau))$.
Let $\sigma^\circ(\tau) \subset \sigma(\tau)$ be the unique lattice up to homothety with cosocle isomorphic to $\sigma$ (see \cite[Lemma 4.1.1]{EGS}).
Let $\overline{\sigma}^\circ(\tau)$ be the reduction of $\sigma^\circ(\tau)$.
Then the natural map $R_\mu \surj \overline{\sigma}^\circ(\tau)$ induces a map 
\begin{equation}\label{eqn:fil2}
R_\mu/\Fil^2_\otimes R_\mu \surj \overline{\sigma}^\circ(\tau)/\rad^2 \overline{\sigma}^\circ(\tau).
\end{equation}
By \cite[Proposition 3.2]{LMS}, the Jordan--H\"older factors of $R_\mu/\Fil^2_\otimes R_\mu$ appear without multiplicity.
Moreover, those Jordan--H\"older factors which are also in $W(\rhobar)$ are in $\JH(\overline{\sigma}^\circ(\tau)/\rad^2 \overline{\sigma}^\circ(\tau))$ by \cite[Theorem 5.1.1]{EGS} (these are exactly the Serre weights $\sigma_J$ with respect to $\mu$  with $J \subset S_\rhobar^\sigma$ and $\# J = 1$.).
Thus the kernel of the map (\ref{eqn:fil2}) contains no Jordan--H\"older factors in $W(\rhobar)$.
We then see that the induced map 
\[
M'_\infty(R_\mu/\Fil^2_\otimes R_\mu) \surj M'_\infty(\overline{\sigma}^\circ(\tau)/\rad^2 \overline{\sigma}^\circ(\tau))
\]
is an isomorphism.
As $M'_\infty(\overline{\sigma}^\circ(\tau))$ is a cyclic $R_\infty$-module by \cite[Theorem 10.1.1]{EGS}, so is $M'_\infty(\overline{\sigma}^\circ(\tau)/\rad^2 \overline{\sigma}^\circ(\tau))$.
\end{proof}

\begin{lemma}\label{lemma:covering}
Suppose that $I \subset S$ such that 
\[
\#(I \cap \{\pm \omega^{(i)}\}) + \#(S_\rhobar^\sigma \cap \{\pm \omega^{(i)}\}) = 1
\]
for all $i$.
Let $N$ be a submodule of $\Fil^k_\otimes R_{\mu,I}/\Fil^{k+2}_\otimes R_{\mu,I}$, and let $\ovl{N}$ be its image in $\gr^k_\otimes R_{\mu,I}$.
If $\gr^k_\otimes R_{\mu,I}/\ovl{N}$ contains no Serre weights in $W(\rhobar)$, then \[(\Fil^k_\otimes R_{\mu,I}/\Fil^{k+2}_\otimes R_{\mu,I})/N\] contains no Jordan--H\"older factors in $W(\rhobar)$.
\end{lemma}
\begin{proof}
It suffices to show that $\gr^{k+1}_\otimes R_{\mu,I}/\gr^{k+1}_\otimes N$ contains no Jordan--H\"older factors in $W(\rhobar)$, since by assumption $\gr^k_\otimes R_{\mu,I}/\gr^k_\otimes N$ contains no Jordan--H\"older factors in $W(\rhobar)$.
In fact, it suffices to show that $\gr^{k+1}_\otimes W_{\bk,\bk+1,I}/(N \cap \gr^{k+1}_\otimes W_{\bk,\bk+1,I})$ contains no Jordan--H\"older factors in $W(\rhobar)$ since $\sum_{|\bk | = k} \gr^{k+1}_\otimes W_{\bk,\bk+1,I} = \gr^{k+1}_\otimes R_{\mu,I}$.

By Proposition \ref{prop:Wk}, a Jordan--H\"older factor of $\gr^{k+1}_\otimes W_{\bk,\bk+1,I}$ has the form $\sigma_{J'}$ with respect to $\mu$ where $J' \cap I = \emptyset$ and there is a $j\in \Z/f$ such that if $\bk(J') = \bk'$ then $k'_i = k_i$ for all $i\neq j$ and $k'_j = k_j+1$.
Suppose that $\sigma_{J'}\in W(\rhobar)$.
If $k'_j = 2$, then let $J = J' \setminus \{-w_j\omega^{(j)}\}$ (with $w$ defined in the beginning of the section).
Otherwise, $J' \cap \{\pm \omega^{(j)}\} = \{w_j\omega^{(j)}\}$ since we assumed that $\sigma_{J'} \in W(\rhobar)$.
In this case, let $J = J' \setminus \{w_j\omega^{(j)}\}$.
Then $\sigma_J \in W(\rhobar)$ and is thus a Jordan--H\"older factor of $N \cap W_{\bk,\bk+1,I}$.
By Proposition \ref{prop:ext}, $\sigma_{J'}$ is a Jordan--H\"older factor of $N$.
\end{proof}

The following lemma generalizes \cite[Lemma 10.1.13]{EGS}, one of the methods used to compute patched modules.

\begin{lemma}\label{lemma:ca}
Let $R$ be a local ring, and $M'' \subset M' \subset M$ be $R$-modules such that $M'/M''$ and $M'$ are minimally generated by the same finite number of elements.
Then $M'' \subset \mathfrak{m}M$.
If, moreover, $M$ is finitely generated over $R$, then $M/M''$ and $M$ are minimally generated by the same number of elements.
\end{lemma}
\begin{proof}
By Nakayama's lemma, that $M'/M''$ and $M'$ are minimally generated by the same finite number of elements implies that $M'' \subset \mathfrak{m}M'$ and thus $M'' \subset \mathfrak{m}M$.
If $M$ is finitely generated, then another application of Nakayama's lemma implies that $M/M''$ and $M$ are minimally generated by the same number of elements.
\end{proof}

The following proposition generalizes the results and methods of \cite{HW,LMS} by combining Lemmas \ref{lemma:patch2}, \ref{lemma:covering}, and \ref{lemma:ca}.

\begin{prop}\label{prop:oldcyc}
Suppose that $I \subset S$ such that $\#(I \cap \{\pm \omega^{(i)}\}) + \#(S_\rhobar^\sigma \cap \{\pm \omega^{(i)}\}) = 1$.
Then $M'_\infty(\tld{R}_{\mu,I})$ is a cyclic $R_\infty$-module.
\end{prop}
\begin{proof}
By Nakayama's lemma, it suffices to show that $M'_\infty(R_{\mu,I})$ is a cyclic $R_\infty$-module.
We will show that $M'_\infty(R_{\mu,I}/\Fil^{k+1}_\otimes R_{\mu,I})$ is a cyclic $R_\infty$-module by induction on $k$.
If $k=1$, then the result follows from Lemma \ref{lemma:patch2}.

Now suppose that $M'_\infty(R_{\mu,I}/\Fil^{k+1}_\otimes R_{\mu,I})$ is a cyclic $R_\infty$-module.
Let $\mathfrak{J}$ be 
\[
\{J\subset S:k(J) = k,J \cap I = \emptyset,\sigma_J \in W(\rhobar)\}.
\]
Recall that for each $J\in \mathfrak{J}$,
\[
\overline{V}_J\subset \Fil^k_\otimes R_\mu/\Fil^{k+2}_\otimes R_\mu
\]
is defined before \cite[Proposition 3.9]{LMS} to be the minimal submodule whose image in $\gr^k_\otimes R_\mu$ contains $\sigma_J$.
Then we let $\overline{V}_{J,I}$ be the image of $\overline{V}_J$ in $R_{\mu,I}/\Fil^{k+2}_\otimes R_{\mu,I}$.
Note that $M'_\infty(\overline{V}_{J,I})$ is a cyclic $R_\infty$-module by Lemma \ref{lemma:patch2}.
Let $\overline{V}$ be $\sum_{J\in \mathfrak{J}} \overline{V}_{J,I} \subset \Fil^k_\otimes R_{\mu,I}/\Fil^{k+2}_\otimes R_{\mu,I}$.
By Lemma \ref{lemma:covering}, the quotient $(\Fil^k_\otimes R_{\mu,I}/\Fil^{k+2}_\otimes R_{\mu,I})/\overline{V}$ does not contain any Jordan--H\"older factors in $W(\rhobar)$.
Thus the natural inclusion $M'_\infty(\overline{V}) \subset M'_\infty(\Fil^k_\otimes R_{\mu,I}/\Fil^{k+2}_\otimes R_{\mu,I})$ is an equality.
In particular, 
\[
M'_\infty(\Fil^k_\otimes R_{\mu,I}/\Fil^{k+2}_\otimes R_{\mu,I})
\]
is generated by no more than $\# \mathfrak{J}$ elements.
On the other hand, $M'_\infty(\gr^k_\otimes R_{\mu,I}) \cong \oplus_{J\in \mathfrak{J}} M'_\infty(\sigma_J)$ is generated by (at least) $\#\mathfrak{J}$ elements.
By Lemma \ref{lemma:ca} with $M = M'_\infty(R_{\mu,I}/\Fil^{k+2}_\otimes R_{\mu,I})$, $M' = M'_\infty(\Fil^k_\otimes R_{\mu,I}/\Fil^{k+2}_\otimes R_{\mu,I})$, and $M'' = M'_\infty(\gr^{k+1}_\otimes R_{\mu,I})$, $M'_\infty(R_{\mu,I}/\Fil^{k+2}_\otimes R_{\mu,I})$ is a cyclic $R_\infty$-module.
\end{proof}

\begin{prop}\label{prop:supp}
The scheme-theoretic support of $M'_\infty(\tld{R}_{\sigma,I})$ is $\Spec (R_\infty \widehat{\otimes}_{R_{\rhobar}^{\psi,\square}} R^{\psi,T_{\sigma,I}})$.
\end{prop}
\begin{proof}
Since $M'_\infty(\tld{R}_{\sigma,I})[p^{-1}]$ is isomorphic to $\oplus_{\sigma(\tau) \in T_{\sigma,I}} M'_\infty(\sigma(\tau))$, the scheme-theoretic support of $M'_\infty(\tld{R}_{\sigma,I})[p^{-1}]$ is $\cup_{\sigma(\tau) \in T_{\sigma,I}} \Spec (R_\infty \widehat{\otimes}_{R_{\rhobar}^{\psi,\square}} R^{\psi,\tau})[p^{-1}]$ by the proof of \cite[Theorem 9.1.1]{EGS}.
Since $M'_\infty(\tld{R}_{\sigma,I})$ is $\cO$-flat by definition of a patching functor, the scheme-theoretic support of $M'_\infty(\tld{R}_{\sigma,I})$ is the Zariski closure of that of $M'_\infty(\tld{R}_{\sigma,I})[p^{-1}]$.
The result now follows from the definition of $\Spec R^{\psi,T_{\sigma,I}}$.
\end{proof}

In order to weaken the hypotheses on $I$ in Proposition \ref{prop:oldcyc}, we compute an integral scheme intersection, of which the following lemma is the key example.

\begin{lemma}\label{lemma:glue}
There is an exact sequence 
\[0 \ra \cO[\![Y]\!]/(Y(Y-p)) \ra \cO[\![Y]\!]/(Y) \oplus \cO[\![Y]\!]/(Y-p) \ra \cO[\![Y]\!]/(Y,p) \ra 0,\] where the second and third maps are the sum and difference, respectively, of the natural projections.
\end{lemma}
\begin{proof}
Given a ring $R$ and ideals $I$ and $J\subset R$, the sequence 
\[0 \ra R/ (I \cap J) \ra R/I \oplus R/J \ra R/(I+J) \ra 0,\]
where the second and third maps are the sum and difference, respectively, of the natural projections, is exact.
The lemma follows from this exact sequence and the relations $(Y) \cap (Y-p) = (Y(Y-p))$ and $(Y)+(Y-p) = (Y,p)$ in $\cO[\![Y]\!]$.
\end{proof}

The following is our main result in the setting of patching functors.
Recall that $\rhobar$ is generic, but not semisimple.

\begin{thm}\label{thm:multone}
Suppose that $I \subset S$ such that $\#(I \cap \{\pm \omega^{(i)}\}) + \#(S_\rhobar^\sigma \cap \{\pm \omega^{(i)}\}) \leq 1$.
Then $M'_\infty(\tld{R}_{\mu,I})$ is a cyclic $R_\infty$-module.
\end{thm}
\begin{proof}
We proceed by induction on $k := f - \#S_\rhobar^\sigma - \#I$.
The case $k=0$ follows from Proposition \ref{prop:oldcyc}.
Suppose that $k>0$ and that $(I\cup S_\rhobar^\sigma)\cap \{\pm\omega^{(j)}\} = \emptyset$.
Then there is an exact sequence
\[0 \ra \tld{R}_{\mu,I} \ra \tld{R}_{\mu,I \cup \{\omega^{(j)}\} } \oplus \tld{R}_{\mu,I \cup \{-\omega^{(j)}\} } \ra R_{\mu,I \cup \{\pm\omega^{(j)}\} } \ra 0,\]
which induces an exact sequence
\[0 \ra M'_\infty(\tld{R}_{\mu,I}) \ra M'_\infty(\tld{R}_{\mu,I \cup \{\omega^{(j)}\} }) \oplus M'_\infty(\tld{R}_{\mu,I \cup \{-\omega^{(j)}\} }) \ra M'_\infty(R_{\mu,I \cup \{\pm\omega^{(j)}\} }) \ra 0,\]
where the third map is the sum of two surjections by exactness of $M'_\infty(\cdot)$.
By the inductive hypothesis and Proposition \ref{prop:supp}, $M'_\infty(\tld{R}_{\mu,I \cup \{\omega^{(j)}\} })$ and $M'_\infty(\tld{R}_{\mu,I \cup \{-\omega^{(j)}\} })$ are cyclic $R_\infty$-modules with scheme-theoretic support $\Spec R_\infty \widehat{\otimes}_{R_{\rhobar}^{\psi,\square}} R^{\psi,T_{\sigma,I \cup \{\omega^{(j)}\} }}$ and $\Spec R_\infty \widehat{\otimes}_{R_{\rhobar}^{\psi,\square}}R^{\psi,T_{\sigma,I \cup \{-\omega^{(j)}\} }}$, respectively.
The scheme-theoretic support of $M'_\infty(R_{\mu,I \cup \{\pm\omega^{(j)}\} })$ is thus a closed subscheme of the intersections of $\Spec R_\infty \widehat{\otimes}_{R_{\rhobar}^{\psi,\square}} R^{\psi,T_{\sigma,I \cup \{\omega^{(j)}\} }}$ and $\Spec R_\infty \widehat{\otimes}_{R_{\rhobar}^{\psi,\square}} R^{\psi,T_{\sigma,I \cup \{-\omega^{(j)}\} }}$, which is $\Spec R_\infty \widehat{\otimes}_{R_{\rhobar}^{\psi,\square}} R^{\psi,T_{\sigma,I \cup \{\omega^{(j)}\} }}/p$ by Theorem \ref{thm:defring} and Lemma \ref{lemma:illset} (we can assume without loss of generality that $\mu$ has the form in \S \ref{sec:defring} by twisting).
Since $M'_\infty(R_{\mu,I \cup \{\pm\omega^{(j)}\} })$ is a cyclic $R_\infty$-module, there is a surjection 
\[R_\infty \widehat{\otimes}_{R_{\rhobar}^{\psi,\square}} R^{\psi,T_{\sigma,I \cup \{\omega^{(j)}\} }}/p \surj M'_\infty(R_{\mu,I \cup \{\pm\omega^{(j)}\} }).\]
Since $\{\pm \omega^{(j)} \} \cap S_\rhobar^\sigma = \emptyset$, from Proposition \ref{prop:Wk} we see that $M'_\infty(R_{\mu,I \cup \{\omega^{(j)}\} })$ and $M'_\infty(R_{\mu,I \cup \{\pm\omega^{(j)}\} })$ have the same Hilbert--Samuel multiplicity.
Thus, both sides of the map $R_\infty \widehat{\otimes}_{R_{\rhobar}^{\psi,\square}} R^{\psi,T_{\sigma,I \cup \{\omega^{(j)}\} }}/p \surj M'_\infty(R_{\mu,I \cup \{\pm\omega^{(j)}\} })$ have the same Hilbert--Samuel multiplicity.
Since $R^{\psi,T_{\sigma,I \cup \{\omega^{(j)}\} }}/p$ contains no embedded primes, this map is an isomorphism (see the argument of \cite[Lemma 6.1.1]{le}).

In summary, there is an exact sequence
\[0 \ra M'_\infty(\tld{R}_{\mu,I}) \ra R_\infty \widehat{\otimes}_{R_{\rhobar}^{\psi,\square}} R^{\psi,T_{\sigma,I \cup \{\omega^{(j)}\} }} \oplus R_\infty \widehat{\otimes}_{R_{\rhobar}^{\psi,\square}} R^{\psi,T_{\sigma,I \cup \{-\omega^{(j)}\} }} \ra R_\infty \widehat{\otimes}_{R_{\rhobar}^{\psi\square}} R^{\psi,T_{\sigma,I \cup \{\omega^{(j)}\} }}/p \ra 0,\]
where the third map is the sum of two surjections.
Any lift of a generator under a surjection between two cyclic modules over a local ring is again a generator by Nakayama's lemma.
Hence, we can assume that the third map is the difference of the natural projections.
Then by Theorem \ref{thm:defring} and Lemma \ref{lemma:illset}, this exact sequence is obtained from taking a completed tensor product with the exact sequence in Lemma \ref{lemma:glue}.
Hence, we see that $M'_\infty(\tld{R}_{\mu,I}) \cong R_\infty \widehat{\otimes}_{R_{\rhobar}^{\psi,\square}} R^{\psi,T_{\sigma,I}}$, and in particular that $M'_\infty(\tld{R}_{\mu,I})$ is a cyclic $R_\infty$-module.
\end{proof}

\section{Global results}\label{sec:main}

Let $F$ be a totally real field in which $p$ is unramified.
Let $D_{/F}$ be a quaternion algebra which is unramified at all places dividing $p$ and at most one infinite place, and let $\rbar:G_F \ra \GL_2(\F)$ be a Galois representation.
If $D_{/F}$ is indefinite and $K = \prod_w K_w \subset (D\otimes_F \A_F^\infty)^\times$ is an open compact subgroup, then there is a smooth projective curve $X_K$ defined over $F$ and we define $S(K,\F)$ to be $H^1((X_K)_{/\overline{F}},\F)$.
If $D_{/F}$ is definite, then we let $S(K,\F)$ be the space of $K$-invariant continuous functions
\[f: D^\times\backslash (D\otimes_F \mathbb{A}_F^\infty)^\times \ra \F.\]
Let $S$ be the union of the set of places in $F$ where $\rbar$ is ramified, the set of places in $F$ where $D$ is ramified, and the set of places in $F$ dividing $p$.
Let $\mathbb{T}^{S,\mathrm{univ}}$ be the commutative polynomial algebra over $\cO$ generated by the formal variables $T_w$ and $S_w$ for each $w \notin S\cup \{w_1\}$ where $w_1$ is chosen as in \cite[\S 6.2]{EGS}.
Then $\mathbb{T}^{S,\mathrm{univ}}$ acts on $S(K,\F)$ with $T_w$ and $S_w$ acting by the usual double coset action of
\[ \big[ \GL_2(\cO_{F_w}) \begin{pmatrix}
  \varpi_w &   \\
   & 1 \\
 \end{pmatrix}\GL_2(\cO_{F_w}) \big] \]
and 
\[ \big[ \GL_2(\cO_{F_w}) \begin{pmatrix}
  \varpi_w &   \\
   & \varpi_w \\
 \end{pmatrix}\GL_2(\cO_{F_w}) \big], \]
respectively.
Let $\mathbb{T}^{S,\mathrm{univ}}\ra \F$ be the map such that the image of $X^2 - T_w X + (\mathbb{N}w) S_w$ in $\F[X]$ is the characteristic polynomial of $\rhobar^\vee(\Frob_w)$, where $\Frob_w$ is a geometric Frobenius element at $w$, and let the kernel be $\fm_{\rbar}$.

For the rest of the section, suppose that
\begin{enumerate}
\item $\rbar$ is modular, i.e. that there exists $K$ such that $S(K,\F)_{\fm_{\rbar}}$ is nonzero;
\item $\rbar|_{G_{F(\zeta_p)}}$ is absolutely irreducible;
\item if $p=5$ then the image of $\rbar(G_{F(\zeta_p)})$ in $\PGL_2(\F)$ is not isomorphic to $A_5$;
\item $\rbar|_{G_{F_w}}$ is generic (Definition \ref{def:gen}) for all places $w|p$; and
\item $\rbar|_{G_{F_w}}$ is non-scalar at all finite places where $D$ ramifies.
\end{enumerate}
Let $v|p$ be a place of $F$, and let $\rhobar$ be $\rbar|_{G_{F_v}}$.
Let $k_v$ be the residue field of $F_v$.

We define $S^{\mathrm{min}}$ to be $S(K^v,\otimes_{w\in S,w\neq v} L_w)_{\fm'_{\rbar}}$ as in \cite[\S 6.5]{EGS}.
We define $M^{\mathrm{min}}$ to be the $\F$-linear dual of $(S^{\mathrm{min}}\otimes_{\cO}\F)[\fm'_{\rbar}]$, factoring out the Galois action in the indefinite case (see \cite[\S 6.2]{EGS}).

\begin{thm}\label{thm:K1multone}
Suppose that $\rbar:G_F \ra \GL_2(\F)$ is a Galois representation satisfying $(1)$-$(5)$.
If $\sigma \in W(\rhobar)$ and $R_\sigma$ is the $\F[\GL_2(k_v)]$-projective envelope of $\sigma$, then $\Hom_{\F[\GL_2(k_v)]}(R_\sigma,(M^{\mathrm{min}})^*)$ is one-dimensional.
\end{thm}
\begin{proof}
The case where $\rhobar$ is semisimple follows from \cite[Corollary 5.4]{LMS}.
We now assume that $\rhobar$ is not semisimple.
Let $\sigma = F(\mu-\eta) \in W(\rhobar)$.
Identify $k_v$ with a finite field $\F_q$.
Then $R_\sigma$ is $R_\mu \otimes_{\F_q} \F$.
Let $M_\infty$ be the minimal fixed determinant patching functor defined in \cite[\S 6.5]{EGS}.
By construction, if $\fm_{R_\infty}$ is the maximal ideal of $R_\infty$, then $\Hom_{\GL_2(\F_q)}(R_\sigma,(M^{\mathrm{min}})^*)$ is the dual of $M_\infty(R_\sigma)/\fm_{R_\infty} = M_\infty'(R_\mu)/\fm_{R_\infty}$, which is one dimensional since $M'_\infty(R_\mu)$ is a cyclic $R_\infty$-module by Theorem \ref{thm:multone}.
\end{proof}

Let $M^{\mathrm{min}}(K_v(1))$ denote the space of coinvariants $(M^{\mathrm{min}})_{K_v(1)}$.
Note that $M^{\mathrm{min}}(K_v(1))$ is isomorphic to the dual of $(S(K^vK_v(1),\otimes_{w\in S,w\neq v} L_w)\otimes_{\cO}\F)[\fm'_{\rbar}]$, factoring out the Galois action in the indefinite case, by a standard spectral sequence argument using that $\fm_{\rbar}'$ is non-Eisenstein.

\begin{cor}\label{cor:main}
Suppose that $\rbar:G_F \ra \GL_2(\F)$ is a Galois representation satisfying $(1)$-$(5)$.
Then the $\GL_2(\F_q)$-representation $(M^{\mathrm{min}}(K_v(1)))^*$ is isomorphic to $D_0(\rhobar)$.
In particular, $(M^{\mathrm{min}}(K_v(1)))^*$ depends only on $\rhobar$ and is multiplicity free.
\end{cor}
\begin{proof}
There is an injection $D_0(\rhobar) \into (M^{\mathrm{min}}(K_v(1)))^*$ by \cite[Proposition 9.3]{breuil}.
Fix an $\F[\GL_2(\F_q)]$-injective hull $(M^{\mathrm{min}}(K_v(1)))^* \into I$.
Since 
\[\Hom_{\GL_2(\F_q)}(R_\sigma,(M^{\mathrm{min}}(K_v(1)))^*)\]
is one-dimensional for all $\sigma \in W(\rhobar)$ by Theorem \ref{thm:K1multone}, this injective hull factors through $D_0(\rhobar)$ by \cite[Theorem 1.1(i)]{BP}.
Since $D_0(\rhobar)$ and $(M^{\mathrm{min}}(K_v(1)))^*$ are finite length $\F[\GL_2(\F_q)]$-modules, they must be isomorphic.
Finally, note that $D_0(\rhobar)$ is multiplicity free by \cite[Theorem 1.1(ii)]{BP}.
\end{proof}

\bibliographystyle{amsalpha}
\bibliography{multonewild2}

\end{document}